\documentclass[12pt]{amsart}

\usepackage[vcentermath,enableskew]{youngtab}
\usepackage{amssymb}
\usepackage{xcolor}
\usepackage[all]{xy}
\usepackage{blkarray}

\usepackage{pgf}
\usepackage{ytableau}
\ytableausetup{centertableaux,nobaseline}

\newcommand{\nopicture}[1]{}

\newcommand{\fk}{\mathfrak{k}}
\newcommand{\fs}{\mathfrak{s}}
\newcommand{\fg}{\mathfrak{g}}
\newcommand{\fp}{\mathfrak{p}}
\newcommand{\fn}{\mathfrak{n}}

\newcommand{\gl}{\mathfrak{gl}}
\newcommand{\fb}{\mathfrak{b}}

\newcommand{\RS}{\mathrm{RS}}
\newcommand{\shape}{\mathrm{shape}}

\newcommand{\genRS}{\mathrm{gRS}}

\newcommand{\card}[1]{\# #1}
\newcommand{\rank}{\mathrm{rank}}
\newcommand{\Span}[1]{\langle #1\rangle}

\newcommand{\nrboxes}[2]{\# #1_{\leq #2}}
\newcommand{\nrplus}[2]{\# #1_{\leq #2}(+)}
\newcommand{\nrminus}[2]{\# #1_{\leq #2}(-)}

\newcommand{\nil}{\mathfrak{nil}}

\newcommand{\Mat}{\mathrm{M}}

\newcommand{\subm}[2]{(#1)_{#2}}

\newcommand{\flag}{\mathcal{F}}

\newcommand{\conormal}{\mathcal{Y}}
\newcommand{\Xorbit}{\mathbb{O}}
\newcommand{\Gorbit}{\mathcal{O}}
\newcommand{\Nkorbit}{\mathfrak{O}}
\newcommand{\Nsorbit}{\mathfrak{O}}

\newcommand{\nilpotentsof}[1]{\mathcal{N}_{#1}}

\newcommand{\condir}{\mathcal{D}}

\newcommand{\intp}{\overline{p}}
\newcommand{\intq}{\overline{q}}

\newcommand{\partitionsof}[1]{\mathcal{P}(#1)}
\newcommand{\permutationsof}[1]{\mathfrak{S}_{#1}}
\newcommand{\ppermutationsof}{\mathfrak{T}}

\newcommand{\parameters}{\overline{\ppermutationsof}}

\newcommand{\columnstrip}{\subset\!\!\!\!\!\cdot\,\,\,}
\newcommand{\smallcolumnstrip}{\subset\!\!\cdot\,\,}

\numberwithin{equation}{section}

\newtheorem{theorem}{Theorem}[section]
\newtheorem{lemma}[theorem]{Lemma}
\newtheorem{proposition}[theorem]{Proposition}
\newtheorem{corollary}[theorem]{Corollary}

\theoremstyle{definition}
\newtheorem{example}[theorem]{Example}

\newtheorem{remark}[theorem]{Remark}

\newcommand{\Xfv}{\mathfrak{X}}
\newcommand{\Grass}{\mathrm{Gr}}
\newcommand{\Flags}{\mathrm{Fl}}
\newcommand{\Lie}{\mathrm{Lie}}
\newcommand{\GL}{\mathrm{GL}}

\newcommand{\C}{\mathbb{C}}

\newcommand{\diag}{\qopname\relax o{diag}}

\newcommand{\graphexA}{\mbox{\tiny
$\begin{picture}(42,20)(0,0)
\put(0,11){$\bullet$}\put(20,11){$\bullet$}\put(40,11){$\bullet$}
\put(0,-11){$\bullet$}\put(20,-11){$\bullet$}
\put(2,12){\line(0,-1){20}}
\put(22,12){\line(0,-1){20}}
\end{picture}$}}

\newcommand{\graphexB}{\mbox{\tiny
$\begin{picture}(42,20)(0,0)
\put(0,11){$\bullet$}\put(20,11){$\bullet$}\put(40,11){$\bullet$}
\put(0,-11){$\bullet$}\put(20,-11){$\bullet$}
\put(2,12){\line(1,-1){20}}
\put(21,12){\line(-1,-1){20}}
\end{picture}$}}

\newcommand{\graphexC}{\mbox{\tiny
$\begin{picture}(42,20)(0,0)
\put(0,11){$\bullet$}\put(20,11){$\bullet$}\put(40,11){$\bullet$}
\put(0,-11){$\bullet$}\put(20,-11){$\bullet$}
\put(21,12){\line(-1,-1){20}}
\put(41,12){\line(-1,-1){20}}
\end{picture}$}}

\newcommand{\graphexD}{\mbox{\tiny
$\begin{picture}(42,20)(0,0)
\put(0,11){$\bullet$}\put(20,11){$\bullet$}\put(40,11){$\bullet$}
\put(0,-11){$\bullet$}\put(20,-11){$\bullet$}
\put(22,12){\line(0,-1){20}}
\put(41,12){\line(-2,-1){40}}
\end{picture}$}}

\newcommand{\graphexE}{\mbox{\tiny
$\begin{picture}(42,20)(0,0)
\put(0,11){$\bullet$}\put(20,11){$\bullet$}\put(40,11){$\bullet$}
\put(0,-11){$\bullet$}\put(20,-11){$\bullet$}
\put(2,12){\line(0,-1){20}}
\put(41,12){\line(-1,-1){20}}
\end{picture}$}}

\newcommand{\graphexF}{\mbox{\tiny
$\begin{picture}(42,20)(0,0)
\put(0,11){$\bullet$}\put(20,11){$\bullet$}\put(40,11){$\bullet$}
\put(0,-11){$\bullet$}\put(20,-11){$\bullet$}
\put(2,12){\line(1,-1){20}}
\put(41,12){\line(-2,-1){40}}
\end{picture}$}}

\newcommand{\graphexG}{\mbox{\tiny
$\begin{picture}(42,20)(0,0)
\put(0,11){$\bullet$}\put(20,11){$\bullet$}\put(40,11){$\bullet$}
\put(0,-11){$\bullet$}\put(20,-11){$\bullet$}
\put(2,12){\line(0,-1){20}}
\put(22.3,13){\circle{8}}
\end{picture}$}}

\newcommand{\graphexH}{\mbox{\tiny
$\begin{picture}(42,20)(0,0)
\put(0,11){$\bullet$}\put(20,11){$\bullet$}\put(40,11){$\bullet$}
\put(0,-11){$\bullet$}\put(20,-11){$\bullet$}
\put(2,12){\line(0,-1){20}}
\put(42.3,13){\circle{8}}
\end{picture}$}}

\newcommand{\graphexI}{\mbox{\tiny
$\begin{picture}(42,20)(0,0)
\put(0,11){$\bullet$}\put(20,11){$\bullet$}\put(40,11){$\bullet$}
\put(0,-11){$\bullet$}\put(20,-11){$\bullet$}
\put(1,12){\line(1,-1){20}}
\put(22.3,13){\circle{8}}
\end{picture}$}}

\newcommand{\graphexJ}{\mbox{\tiny
$\begin{picture}(42,20)(0,0)
\put(0,11){$\bullet$}\put(20,11){$\bullet$}\put(40,11){$\bullet$}
\put(0,-11){$\bullet$}\put(20,-11){$\bullet$}
\put(1,12){\line(1,-1){20}}
\put(42.3,13){\circle{8}}
\end{picture}$}}

\newcommand{\graphexK}{\mbox{\tiny
$\begin{picture}(42,20)(0,0)
\put(0,11){$\bullet$}\put(20,11){$\bullet$}\put(40,11){$\bullet$}
\put(0,-11){$\bullet$}\put(20,-11){$\bullet$}
\put(21,12){\line(-1,-1){20}}
\put(2.3,13){\circle{8}}
\end{picture}$}}

\newcommand{\graphexL}{\mbox{\tiny
$\begin{picture}(42,20)(0,0)
\put(0,11){$\bullet$}\put(20,11){$\bullet$}\put(40,11){$\bullet$}
\put(0,-11){$\bullet$}\put(20,-11){$\bullet$}
\put(21,12){\line(-1,-1){20}}
\put(42.3,13){\circle{8}}
\end{picture}$}}

\newcommand{\graphexM}{\mbox{\tiny
$\begin{picture}(42,20)(0,0)
\put(0,11){$\bullet$}\put(20,11){$\bullet$}\put(40,11){$\bullet$}
\put(0,-11){$\bullet$}\put(20,-11){$\bullet$}
\put(22,12){\line(0,-1){20}}
\put(2.3,13){\circle{8}}
\end{picture}$}}

\newcommand{\graphexN}{\mbox{\tiny
$\begin{picture}(42,20)(0,0)
\put(0,11){$\bullet$}\put(20,11){$\bullet$}\put(40,11){$\bullet$}
\put(0,-11){$\bullet$}\put(20,-11){$\bullet$}
\put(22,12){\line(0,-1){20}}
\put(42.3,13){\circle{8}}
\end{picture}$}}

\newcommand{\graphexO}{\mbox{\tiny
$\begin{picture}(42,20)(0,0)
\put(0,11){$\bullet$}\put(20,11){$\bullet$}\put(40,11){$\bullet$}
\put(0,-11){$\bullet$}\put(20,-11){$\bullet$}
\put(41,12){\line(-2,-1){40}}
\put(2.3,13){\circle{8}}
\end{picture}$}}

\newcommand{\graphexP}{\mbox{\tiny
$\begin{picture}(42,20)(0,0)
\put(0,11){$\bullet$}\put(20,11){$\bullet$}\put(40,11){$\bullet$}
\put(0,-11){$\bullet$}\put(20,-11){$\bullet$}
\put(41,12){\line(-2,-1){40}}
\put(22.3,13){\circle{8}}
\end{picture}$}}

\newcommand{\graphexQ}{\mbox{\tiny
$\begin{picture}(42,20)(0,0)
\put(0,11){$\bullet$}\put(20,11){$\bullet$}\put(40,11){$\bullet$}
\put(0,-11){$\bullet$}\put(20,-11){$\bullet$}
\put(41,12){\line(-1,-1){20}}
\put(2.3,13){\circle{8}}
\end{picture}$}}

\newcommand{\graphexR}{\mbox{\tiny
$\begin{picture}(42,20)(0,0)
\put(0,11){$\bullet$}\put(20,11){$\bullet$}\put(40,11){$\bullet$}
\put(0,-11){$\bullet$}\put(20,-11){$\bullet$}
\put(41,12){\line(-1,-1){20}}
\put(22.3,13){\circle{8}}
\end{picture}$}}

\newcommand{\graphexS}{\mbox{\tiny
$\begin{picture}(42,20)(0,0)
\put(0,11){$\bullet$}\put(20,11){$\bullet$}\put(40,11){$\bullet$}
\put(0,-11){$\bullet$}\put(20,-11){$\bullet$}
\put(2,12){\line(0,-1){20}}
\put(22.3,-8.6){\circle{8}}
\end{picture}$}}

\newcommand{\graphexT}{\mbox{\tiny
$\begin{picture}(42,20)(0,0)
\put(0,11){$\bullet$}\put(20,11){$\bullet$}\put(40,11){$\bullet$}
\put(0,-11){$\bullet$}\put(20,-11){$\bullet$}
\put(2,12){\line(1,-1){20}}
\put(2.3,-8.6){\circle{8}}
\end{picture}$}}

\newcommand{\graphexU}{\mbox{\tiny
$\begin{picture}(42,20)(0,0)
\put(0,11){$\bullet$}\put(20,11){$\bullet$}\put(40,11){$\bullet$}
\put(0,-11){$\bullet$}\put(20,-11){$\bullet$}
\put(21,12){\line(0,-1){20}}
\put(2.3,-8.6){\circle{8}}
\end{picture}$}}

\newcommand{\graphexV}{\mbox{\tiny
$\begin{picture}(42,20)(0,0)
\put(0,11){$\bullet$}\put(20,11){$\bullet$}\put(40,11){$\bullet$}
\put(0,-11){$\bullet$}\put(20,-11){$\bullet$}
\put(21,12){\line(-1,-1){20}}
\put(22.3,-8.6){\circle{8}}
\end{picture}$}}

\newcommand{\graphexW}{\mbox{\tiny
$\begin{picture}(42,20)(0,0)
\put(0,11){$\bullet$}\put(20,11){$\bullet$}\put(40,11){$\bullet$}
\put(0,-11){$\bullet$}\put(20,-11){$\bullet$}
\put(41,12){\line(-2,-1){40}}
\put(22.3,-8.6){\circle{8}}
\end{picture}$}}

\newcommand{\graphexX}{\mbox{\tiny
$\begin{picture}(42,20)(0,0)
\put(0,11){$\bullet$}\put(20,11){$\bullet$}\put(40,11){$\bullet$}
\put(0,-11){$\bullet$}\put(20,-11){$\bullet$}
\put(41,12){\line(-1,-1){20}}
\put(2.3,-8.6){\circle{8}}
\end{picture}$}}

\newcommand{\graphexY}{\mbox{\tiny
$\begin{picture}(42,20)(0,0)
\put(0,11){$\bullet$}\put(20,11){$\bullet$}\put(40,11){$\bullet$}
\put(0,-11){$\bullet$}\put(20,-11){$\bullet$}
\put(2.3,13){\circle{8}}
\put(22.3,13){\circle{8}}
\end{picture}$}}

\newcommand{\graphexZ}{\mbox{\tiny
$\begin{picture}(42,20)(0,0)
\put(0,11){$\bullet$}\put(20,11){$\bullet$}\put(40,11){$\bullet$}
\put(0,-11){$\bullet$}\put(20,-11){$\bullet$}
\put(42.3,13){\circle{8}}
\put(2.3,13){\circle{8}}
\end{picture}$}}

\newcommand{\graphexAA}{\mbox{\tiny
$\begin{picture}(42,20)(0,0)
\put(0,11){$\bullet$}\put(20,11){$\bullet$}\put(40,11){$\bullet$}
\put(0,-11){$\bullet$}\put(20,-11){$\bullet$}
\put(22.3,13){\circle{8}}
\put(42.3,13){\circle{8}}
\end{picture}$}}

\newcommand{\graphexBB}{\mbox{\tiny
$\begin{picture}(42,20)(0,0)
\put(0,11){$\bullet$}\put(20,11){$\bullet$}\put(40,11){$\bullet$}
\put(0,-11){$\bullet$}\put(20,-11){$\bullet$}
\put(2.3,13){\circle{8}}
\put(2.3,-8.6){\circle{8}}
\end{picture}$}}

\newcommand{\graphexCC}{\mbox{\tiny
$\begin{picture}(42,20)(0,0)
\put(0,11){$\bullet$}\put(20,11){$\bullet$}\put(40,11){$\bullet$}
\put(0,-11){$\bullet$}\put(20,-11){$\bullet$}
\put(2.3,13){\circle{8}}
\put(22.3,-8.6){\circle{8}}
\end{picture}$}}

\newcommand{\graphexDD}{\mbox{\tiny
$\begin{picture}(42,20)(0,0)
\put(0,11){$\bullet$}\put(20,11){$\bullet$}\put(40,11){$\bullet$}
\put(0,-11){$\bullet$}\put(20,-11){$\bullet$}
\put(22.3,13){\circle{8}}
\put(2.3,-8.6){\circle{8}}
\end{picture}$}}

\newcommand{\graphexEE}{\mbox{\tiny
$\begin{picture}(42,20)(0,0)
\put(0,11){$\bullet$}\put(20,11){$\bullet$}\put(40,11){$\bullet$}
\put(0,-11){$\bullet$}\put(20,-11){$\bullet$}
\put(22.3,13){\circle{8}}
\put(22.3,-8.6){\circle{8}}
\end{picture}$}}

\newcommand{\graphexFF}{\mbox{\tiny
$\begin{picture}(42,20)(0,0)
\put(0,11){$\bullet$}\put(20,11){$\bullet$}\put(40,11){$\bullet$}
\put(0,-11){$\bullet$}\put(20,-11){$\bullet$}
\put(42.3,13){\circle{8}}
\put(2.3,-8.6){\circle{8}}
\end{picture}$}}

\newcommand{\graphexGG}{\mbox{\tiny
$\begin{picture}(42,20)(0,0)
\put(0,11){$\bullet$}\put(20,11){$\bullet$}\put(40,11){$\bullet$}
\put(0,-11){$\bullet$}\put(20,-11){$\bullet$}
\put(42.3,13){\circle{8}}
\put(22.3,-8.6){\circle{8}}
\end{picture}$}}

\newcommand{\graphexHH}{\mbox{\tiny
$\begin{picture}(42,20)(0,0)
\put(0,11){$\bullet$}\put(20,11){$\bullet$}\put(40,11){$\bullet$}
\put(0,-11){$\bullet$}\put(20,-11){$\bullet$}
\put(22.3,-8.6){\circle{8}}
\put(2.3,-8.6){\circle{8}}
\end{picture}$}}

\newcommand{\graphnA}{\mbox{\tiny
$\begin{picture}(30,20)(0,0)
\put(0,11){$\bullet$}\put(20,11){$\bullet$}\put(40,11){$\bullet$}
\put(0,16){$1^+$}\put(20,16){$2^+$}\put(40,16){$3^+$}
\put(0,-11){$\bullet$}\put(20,-11){$\bullet$}\put(40,-11){$\bullet$}
\put(0,-18){$1^-$}\put(20,-18){$2^-$}\put(40,-18){$3^-$}
\put(2,12){\line(2,-1){40}}
\put(41,12){\line(-2,-1){40}}
\end{picture}$}}

\newcommand{\graphnB}{\mbox{\tiny
$\begin{picture}(30,20)(0,0)
\put(0,11){$\bullet$}\put(20,11){$\bullet$}\put(40,11){$\bullet$}
\put(0,16){$1^+$}\put(20,16){$2^+$}\put(40,16){$3^+$}
\put(0,-11){$\bullet$}\put(20,-11){$\bullet$}\put(40,-11){$\bullet$}
\put(0,-18){$1^-$}\put(20,-18){$2^-$}\put(40,-18){$3^-$}
\put(2,12){\line(0,-1){20}}
\put(41,12){\line(0,-1){20}}
\end{picture}$}}

\newcommand{\coverA}{\mbox{\tiny
$\begin{picture}(30,20)(0,0)
\put(0,11){$\bullet$}\put(20,11){$\bullet$}
\put(0,17){$a^+$}\put(20,17){$b^+$}
\put(0,-11){$\bullet$}\put(20,-11){$\bullet$}
\put(0,-19){$c^-$}\put(20,-19){$d^-$}
\put(2,12){\line(1,-1){20}}
\put(21,12){\line(-1,-1){20}}
\end{picture}$}}

\newcommand{\coverB}{\mbox{\tiny
$\begin{picture}(30,20)(0,0)
\put(0,11){$\bullet$}\put(20,11){$\bullet$}
\put(0,17){$a^+$}\put(20,17){$b^+$}
\put(0,-11){$\bullet$}\put(20,-11){$\bullet$}
\put(0,-19){$c^-$}\put(20,-19){$d^-$}
\put(2,12){\line(0,-1){20}}
\put(22,12){\line(0,-1){20}}
\end{picture}$}}

\newcommand{\coverC}{\mbox{\tiny
$\begin{picture}(30,20)(0,0)
\put(0,11){$\bullet$}\put(20,11){$\bullet$}
\put(0,17){$a^+$}\put(20,17){$b^+$}
\put(0,-11){$\bullet$}
\put(0,-19){$c^-$}
\put(2,12){\line(0,-1){20}}
\put(22,12){\circle{8}}
\end{picture}$}}

\newcommand{\coverD}{\mbox{\tiny
$\begin{picture}(30,20)(0,0)
\put(0,11){$\bullet$}\put(20,11){$\bullet$}
\put(0,17){$a^+$}\put(20,17){$b^+$}
\put(0,-11){$\bullet$}
\put(0,-19){$c^-$}
\put(21,12){\line(-1,-1){20}}
\put(2,12){\circle{8}}
\end{picture}$}}

\newcommand{\coverE}{\mbox{\tiny
$\begin{picture}(30,20)(0,0)
\put(0,11){$\bullet$}
\put(0,17){$a^+$}
\put(0,-11){$\bullet$}\put(20,-11){$\bullet$}
\put(0,-19){$c^-$}\put(20,-19){$d^-$}
\put(2,12){\line(0,-1){20}}
\put(22,-9){\circle{8}}
\end{picture}$}}

\newcommand{\coverF}{\mbox{\tiny
$\begin{picture}(30,20)(0,0)
\put(0,11){$\bullet$}
\put(0,17){$a^+$}
\put(0,-11){$\bullet$}\put(20,-11){$\bullet$}
\put(0,-19){$c^-$}\put(20,-19){$d^-$}
\put(2,12){\line(1,-1){20}}
\put(2,-9){\circle{8}}
\end{picture}$}}

\newcommand{\coverG}{\mbox{\tiny
$\begin{picture}(30,20)(0,0)
\put(0,11){$\bullet$}\put(20,11){$\bullet$}
\put(0,17){$a^+$}\put(20,17){$b^+$}
\put(0,-11){$\bullet$}
\put(0,-19){$c^-$}
\put(21,12){\line(-1,-1){20}}
\end{picture}$}}

\newcommand{\coverH}{\mbox{\tiny
$\begin{picture}(30,20)(0,0)
\put(0,11){$\bullet$}\put(20,11){$\bullet$}
\put(0,17){$a^+$}\put(20,17){$b^+$}
\put(0,-11){$\bullet$}
\put(0,-19){$c^-$}
\put(2,12){\line(0,-1){20}}
\end{picture}$}}

\newcommand{\coverI}{\mbox{\tiny
$\begin{picture}(30,20)(0,0)
\put(0,11){$\bullet$}
\put(0,17){$a^+$}
\put(0,-11){$\bullet$}\put(20,-11){$\bullet$}
\put(0,-19){$c^-$}\put(20,-19){$d^-$}
\put(2,12){\line(1,-1){20}}
\end{picture}$}}

\newcommand{\coverJ}{\mbox{\tiny
$\begin{picture}(30,20)(0,0)
\put(0,11){$\bullet$}
\put(0,17){$a^+$}
\put(0,-11){$\bullet$}\put(20,-11){$\bullet$}
\put(0,-19){$c^-$}\put(20,-19){$d^-$}
\put(2,12){\line(0,-1){20}}
\end{picture}$}}

\newcommand{\coverK}{\mbox{\tiny
$\begin{picture}(10,20)(0,0)
\put(0,11){$\bullet$}
\put(0,17){$a^+$}
\put(0,-11){$\bullet$}
\put(0,-19){$c^-$}
\put(2,12){\line(0,-1){20}}
\end{picture}$}}

\newcommand{\coverL}{\mbox{\tiny
$\begin{picture}(10,20)(0,0)
\put(0,11){$\bullet$}
\put(0,17){$a^+$}
\put(0,-11){$\bullet$}
\put(0,-19){$c^-$}
\put(2,12){\circle{8}}
\end{picture}$}}

\newcommand{\coverM}{\mbox{\tiny
$\begin{picture}(10,20)(0,0)
\put(0,11){$\bullet$}
\put(0,17){$a^+$}
\put(0,-11){$\bullet$}
\put(0,-19){$c^-$}
\put(2,12){\line(0,-1){20}}
\end{picture}$}}

\newcommand{\coverN}{\mbox{\tiny
$\begin{picture}(10,20)(0,0)
\put(0,11){$\bullet$}
\put(0,17){$a^+$}
\put(0,-11){$\bullet$}
\put(0,-19){$c^-$}
\put(2,-9){\circle{8}}
\end{picture}$}}

\newcommand{\coverO}{\mbox{\tiny
$\begin{picture}(30,20)(0,0)
\put(0,11){$\bullet$}\put(20,11){$\bullet$}
\put(0,17){$a^+$}\put(20,17){$b^+$}
\put(22,12){\circle{8}}
\end{picture}$}}

\newcommand{\coverP}{\mbox{\tiny
$\begin{picture}(30,20)(0,0)
\put(0,11){$\bullet$}\put(20,11){$\bullet$}
\put(0,17){$a^+$}\put(20,17){$b^+$}
\put(2,12){\circle{8}}
\end{picture}$}}

\newcommand{\coverQ}{\mbox{\tiny
$\begin{picture}(30,20)(0,0)
\put(0,-11){$\bullet$}\put(20,-11){$\bullet$}
\put(0,-19){$c^-$}\put(20,-19){$d^-$}
\put(22,-9){\circle{8}}
\end{picture}$}}

\newcommand{\coverR}{\mbox{\tiny
$\begin{picture}(30,20)(0,0)
\put(0,-11){$\bullet$}\put(20,-11){$\bullet$}
\put(0,-19){$c^-$}\put(20,-19){$d^-$}
\put(2,-9){\circle{8}}
\end{picture}$}}

\newcommand{\graphA}{\textcolor{black}{\mbox{\tiny $\begin{picture}(24,24)(0,-9)
\put(0,10){$\bullet$}\put(20,10){$\bullet$}
\put(0,-10){$\bullet$}\put(20,-10){$\bullet$}
\put(2,12){\line(1,-1){20}}
\put(22,12){\line(-1,-1){20}}
\end{picture}$}}}

\newcommand{\graphC}{\textcolor{black}{\mbox{\tiny $\begin{picture}(24,24)(0,-9)
\put(0,10){$\bullet$}\put(20,10){$\bullet$}
\put(0,-10){$\bullet$}\put(20,-10){$\bullet$}
\put(2,12){\line(0,-1){20}}
\put(22,12){\line(0,-1){20}}
\end{picture}$}}}

\newcommand{\graphB}{\textcolor{black}{\mbox{\tiny $\begin{picture}(24,24)(0,-9)
\put(0,10){$\bullet$}\put(20,10){$\bullet$}
\put(0,-10){$\bullet$}\put(20,-10){$\bullet$}
\put(2,12){\line(1,-1){20}}
\put(22.3,11.7){\circle{8}}
\end{picture}$}}}

\newcommand{\graphD}{\textcolor{black}{\mbox{\tiny $\begin{picture}(24,24)(0,-9)
\put(0,10){$\bullet$}\put(20,10){$\bullet$}
\put(0,-10){$\bullet$}\put(20,-10){$\bullet$}
\put(22,12){\line(-1,-1){20}}
\put(22.3,-8){\circle{8}}
\end{picture}$}}}

\newcommand{\graphE}{\textcolor{black}{\mbox{\tiny $\begin{picture}(24,24)(0,-9)
\put(0,10){$\bullet$}\put(20,10){$\bullet$}
\put(0,-10){$\bullet$}\put(20,-10){$\bullet$}
\put(2,12){\line(0,-1){20}}
\put(22.3,11.7){\circle{8}}
\end{picture}$}}}

\newcommand{\graphF}{\textcolor{black}{\mbox{\tiny $\begin{picture}(24,24)(0,-9)
\put(0,10){$\bullet$}\put(20,10){$\bullet$}
\put(0,-10){$\bullet$}\put(20,-10){$\bullet$}
\put(22,12){\line(0,-1){20}}
\put(2.3,11.7){\circle{8}}
\end{picture}$}}}

\newcommand{\graphG}{\textcolor{black}{\mbox{\tiny $\begin{picture}(24,24)(0,-9)
\put(0,10){$\bullet$}\put(20,10){$\bullet$}
\put(0,-10){$\bullet$}\put(20,-10){$\bullet$}
\put(22.3,11.7){\circle{8}}
\put(2.32,-8){\circle{8}}
\end{picture}$}}}

\newcommand{\graphH}{\textcolor{black}{\mbox{\tiny $\begin{picture}(24,24)(0,-9)
\put(0,10){$\bullet$}\put(20,10){$\bullet$}
\put(0,-10){$\bullet$}\put(20,-10){$\bullet$}
\put(22,12){\line(0,-1){20}}
\put(2.3,-8){\circle{8}}
\end{picture}$}}}

\newcommand{\graphI}{\textcolor{black}{\mbox{\tiny $\begin{picture}(24,24)(0,-9)
\put(0,10){$\bullet$}\put(20,10){$\bullet$}
\put(0,-10){$\bullet$}\put(20,-10){$\bullet$}
\put(2,12){\line(0,-1){20}}
\put(22.3,-8){\circle{8}}
\end{picture}$}}}

\newcommand{\graphJ}{\textcolor{black}{\mbox{\tiny $\begin{picture}(24,24)(0,-9)
\put(0,10){$\bullet$}\put(20,10){$\bullet$}
\put(0,-10){$\bullet$}\put(20,-10){$\bullet$}
\put(22,12){\line(-1,-1){20}}
\put(2.3,11.7){\circle{8}}
\end{picture}$}}}

\newcommand{\graphK}{\textcolor{black}{\mbox{\tiny $\begin{picture}(24,24)(0,-9)
\put(0,10){$\bullet$}\put(20,10){$\bullet$}
\put(0,-10){$\bullet$}\put(20,-10){$\bullet$}
\put(22.3,11.7){\circle{8}}
\put(2.3,-8){\circle{8}}
\end{picture}$}}}

\newcommand{\graphL}{\textcolor{black}{\mbox{\tiny $\begin{picture}(24,24)(0,-9)
\put(0,10){$\bullet$}\put(20,10){$\bullet$}
\put(0,-10){$\bullet$}\put(20,-10){$\bullet$}
\put(2.3,11.7){\circle{8}}
\put(22.3,-8){\circle{8}}
\end{picture}$}}}

\newcommand{\graphM}{\textcolor{black}{\mbox{\tiny $\begin{picture}(24,24)(0,-9)
\put(0,10){$\bullet$}\put(20,10){$\bullet$}
\put(0,-10){$\bullet$}\put(20,-10){$\bullet$}
\put(2,12){\line(1,-1){20}}
\put(2.3,-8){\circle{8}}
\end{picture}$}}}

\newcommand{\graphN}{\textcolor{black}{\mbox{\tiny $\begin{picture}(24,24)(0,-9)
\put(0,10){$\bullet$}\put(20,10){$\bullet$}
\put(0,-10){$\bullet$}\put(20,-10){$\bullet$}
\put(2.3,11.7){\circle{8}}
\put(22.3,11.7){\circle{8}}
\end{picture}$}}}

\newcommand{\graphO}{\textcolor{black}{\mbox{\tiny $\begin{picture}(24,24)(0,-9)
\put(0,10){$\bullet$}\put(20,10){$\bullet$}
\put(0,-10){$\bullet$}\put(20,-10){$\bullet$}
\put(2.3,11.7){\circle{8}}
\put(2.3,-8){\circle{8}}
\end{picture}$}}}

\newcommand{\graphP}{\textcolor{black}{\mbox{\tiny $\begin{picture}(24,24)(0,-9)
\put(0,10){$\bullet$}\put(20,10){$\bullet$}
\put(0,-10){$\bullet$}\put(20,-10){$\bullet$}
\put(2.3,-8){\circle{8}}
\put(22.3,-8){\circle{8}}
\end{picture}$}}}

\begin{document}

\title{On generalized Steinberg theory for type AIII}
\author{Lucas Fresse and Kyo Nishiyama}
\address{Universit\'e de Lorraine, CNRS, Institut \'Elie Cartan de Lorraine, UMR 7502, Vandoeu\-vre-l\`es-Nancy, F-54506, France}
\email{lucas.fresse@univ-lorraine.fr}
\address{Department of
Mathematics, Aoyama Gakuin University, Fuchinobe 5-10-1, Chuo-ku,
Sagamihara 252-5258, Japan}
\email{kyo@gem.aoyama.ac.jp}

\begin{abstract}
The multiple flag variety
$\Xfv=\mathrm{Gr}(\mathbb{C}^{p+q},r)\times(\mathrm{Fl}(\mathbb{C}^p)\times\mathrm{Fl}(\mathbb{C}^q))$
can be considered as a double flag variety associated to the symmetric pair $(G,K)=(\mathrm{GL}_{p+q}(\mathbb{C}),$ $\mathrm{GL}_{p}(\mathbb{C})\times \mathrm{GL}_{q}(\mathbb{C}))$ of type AIII.
We consider the diagonal action of $K$ on $\Xfv$.
There is a finite number of orbits for this action, and our first result is a description of these orbits: parametrization (by a certain set of graphs), dimensions, closure relations and cover relations.

In \cite{FN2}, we defined two generalized Steinberg maps from the $K$-orbits of $\Xfv$ to the nilpotent $K$-orbits in $\mathfrak{k}$ and those in the Cartan complement of $\mathfrak{k}$, respectively.
The main result in the present paper is a complete, explicit description of these two Steinberg maps
by means of a combinatorial algorithm
which extends the classical Robinson--Schensted correspondence.
\end{abstract}



\maketitle

\setcounter{tocdepth}{2}

\tableofcontents

\section{Introduction}

\subsection{A multiple flag variety and its orbital decomposition}
\label{section-1.1}
In this paper, we consider the multiple flag variety
\begin{equation}
\label{equation-dfv}
\Xfv=\Grass(V,r)\times \Flags(V^+)\times \Flags(V^-),
\end{equation}
where
\begin{itemize}
\item $V=\C^{p+q}$ is equipped with a polar decomposition $V=V^+\oplus V^-$ with $V^+=\C^p\times\{0\}^q$ and $V^-=\{0\}^p\times \C^q$;
\item $\Grass(V,r)$ denotes the Grassmann variety of $r$-dimensional subspaces of $V$;
\item $\Flags(V^+)$ and $\Flags(V^-)$ denote the varieties of complete flags of $V^+$ and $V^-$, respectively.
\end{itemize}
Each factor of the variety $\Xfv$ has a natural action of
\[
K:=\GL(V^+)\times \GL(V^-)=\left\{\begin{pmatrix} a & 0 \\ 0 & d \end{pmatrix}:a\in\GL_p(\C),\ d\in\GL_q(\C)\right\}\subset\GL(V),
\]
and $\Xfv$ is endowed with the resulting diagonal action of $K$.

The multiple flag variety $\Xfv$ can be written in the form
\[
\Xfv=G/P\times K/B_K,
\]
where $G=\GL(V)$, $P\subset G$ is a maximal parabolic subgroup, $B_K\subset K$ is a Borel subgroup.
In this way, $\Xfv$ is a \textit{double flag variety associated to the symmetric pair}
$(G,K)$
in the sense of \cite{NO} and \cite{HNOO}.
In particular, it is known from \cite{NO} that the above variety $\Xfv$ has a finite number of $K$-orbits.

In \cite{FN1,FN2}, we have initiated an analogue of Steinberg theory for double flag varieties associated to symmetric pairs such as $\Xfv$.
Specifically, we have defined two Steinberg maps, from the set of $K$-orbits of $\Xfv$ to the sets of nilpotent $K$-orbits of $\fk:=\Lie(K)$ and of its Cartan complement $\mathfrak{s}$, respectively.

For general symmetric pairs, the calculation of the generalized Steinberg maps appears to be quite difficult.
In \cite{FN2}, we have considered the variety $\Xfv$ of (\ref{equation-dfv}) in the special case where $p=q=r$ and we have computed the Steinberg maps on a special subset of $K$-orbits parametrized by partial permutations.
The present paper deals with the variety $\Xfv$ of (\ref{equation-dfv}) and its $K$-orbits in full generality.

Here we summarize the main results achieved in this paper:
\begin{itemize}
\item We describe completely the decomposition of $\Xfv$ into $K$-orbits: we give a para\-metrization of the orbits (in terms of certain graphs), we provide a dimension formula, and describe the closure relations and the cover relations; see Section \ref{section-2.2}.
\item Our main result (Theorem \ref{T2}) is the calculation of the two Steinberg maps mentioned above. This calculation is concrete, and it is done by means of a sophisticated combinatorial algorithm that generalizes the Robinson--Schensted correspondence; see Sections \ref{section-2.3}--\ref{section-2.4}.
\end{itemize}
In the following subsection, we explain the construction of the generalized Steinberg maps and we give more insight on our main result.

\subsection{Conormal variety and Steinberg maps}
\label{section-1.2}
We consider the Lie algebras
\[
\fg=\gl_{p+q}(\C)\supset\fk=\left\{\begin{pmatrix} a & 0 \\ 0 & d \end{pmatrix}:a\in\gl_p(\C),\ d\in\gl_q(\C)\right\}
\]
and a Cartan decomposition
\[
\fg=\fk\oplus\fs\qquad\mbox{where}\qquad\fs=\left\{\begin{pmatrix} 0 & b \\ c & 0 \end{pmatrix}:b\in\Mat_{p,q}(\C),\ c\in\Mat_{q,p}(\C)\right\}.
\]
We write $x=x_\fk+x_\fs$ with $(x_\fk,x_\fs)\in\fk\times\fs$ for the decomposition of an element $x\in\fg$ along the Cartan decomposition.
Moreover, we identify the Lie algebras $\fg$ and $\fk$ with their duals $\fg^*$ and $\fk^*$ through the trace form.

Any partial flag $\flag=(F_0=0\subset F_1\subset\ldots\subset F_k=V')$ of a vector space $V'$ gives rise to a parabolic subalgebra $\fp(\flag)$ of $\gl(V')$ and the corresponding nilradical $\nil(\flag)$ defined by
\begin{eqnarray*}
\fp(\flag) & = & \mathrm{Stab}_{\gl(V')}(\flag)=\{x\in\gl(V'):x(F_i)\subset F_i\quad \forall i=1,\ldots,k\}, \\
\nil(\flag) & = & \{x\in\gl(V'):x(F_i)\subset F_{i-1}\quad\forall i=1,\ldots,k\}.
\end{eqnarray*}
For a subspace $W\subset V'$, we denote by $\fp(W)$ and $\nil(W)$ the (maximal) parabolic subalgebra and the nilradical associated to the partial flag $(0\subset W\subset V')$.

As explained in \cite[\S3]{FN2}, the cotangent bundle $T^*\Xfv$ inherits a Hamiltonian action of $K$, which gives rise to a moment map $\mu_\Xfv:T^*\Xfv\to\fk^*=\fk$.
The nullfiber $\conormal=\mu_\Xfv^{-1}(0)$ is called a {\em conormal variety}. It can be described explicitly as
\[
\conormal=\{(W,\flag^+,\flag^-,x)\in \Xfv\times\gl(V):x\in\nil(W),\ x_\fk\in\nil(\flag^+)\times\nil(\flag^-)\}.
\]

Every $K$-orbit $\Xorbit\subset \Xfv$ yields a conormal bundle $T^*_\Xorbit\Xfv$ that can be realized as a (locally-closed) subvariety of $\conormal$ given by
\[
T^*_\Xorbit\Xfv=\{(W,\flag^+,\flag^-,x)\in\conormal:(W,\flag^+,\flag^-)\in\Xorbit\}.
\]
The variety $\conormal$ is equidimensional of dimension $\dim\Xfv$, and its irreducible components are precisely the closures of the various conormal bundles $T^*_\Xorbit\Xfv$, since the set of orbits $\Xfv/K$ is finite.
One can find a more comprehensive introduction to the theory of conormal varieties in \cite[\S 3]{FN2} (see also \cite{CG}).

The conormal variety $\conormal$ is equipped with two $K$-equivariant projections to $\fk$ and $\fs$, namely
\[\phi_\fk:\conormal\to\fk,\ (W,\flag^+,\flag^-,x)\mapsto x_\fk\quad\mbox{and}\quad
\phi_\fs:\conormal\to\fs,\ (W,\flag^+,\flag^-,x)\mapsto x_\fs.\]
It immediately follows from the description of the conormal variety that the image of $\phi_\fk$ is contained in the cone of nilpotent elements $\nilpotentsof{\fk}\subset\fk$.
It is shown in \cite[Proposition 4.2]{FN1} that (for the variety $\Xfv$ considered in this paper) the image of $\phi_\fs$ is also contained in the nilpotent cone $\nilpotentsof{\fs}\subset\fs$.
It is known that both nilpotent cones $\nilpotentsof{\fk}$ and $\nilpotentsof{\fs}$ consist of finitely many adjoint $K$-orbits

Therefore, we can define two maps
\[\Phi_\fk:\Xfv/K\to\nilpotentsof{\fk}/K
\quad\mbox{and}\quad
\Phi_\fs:\Xfv/K\to\nilpotentsof{\fs}/K\]
in the following way: for every orbit $\Xorbit\in\Xfv/K$, define $\Phi_\fk(\Xorbit)\in \nilpotentsof{\fk}/K$, resp. $\Phi_\fs(\Xorbit)\in\nilpotentsof{\fs}/K$, as the unique nilpotent $K$-orbit which is open and dense in the image of the conormal bundle $T^*_\Xorbit\Xfv$ by the projection map $\phi_\fk$, resp. $\phi_\fs$. According to the terminology introduced in \cite{FN2}, we will refer to $\Phi_\fk$ as the {\em symmetrized Steinberg map} and to $\Phi_\fs$ as the {\em exotic Steinberg map}.

In Section \ref{section-2.3}, we describe the maps $\Phi_\fk$ and $\Phi_\fs$. In \cite{FN2}, in the special case where $p=q=r$,
the images $\Phi_\fk(\Xorbit)$ and $\Phi_\fs(\Xorbit)$ are determined when an orbit $\Xorbit$ is contained in a ``big cell'' of $\Xfv/K$.
Moreover, in Section \ref{section-2.4}, we describe the fibers of $\Phi_\fk$ by means of a combinatorial procedure that extends the classical Robinson--Schensted correspondence; this also generalizes \cite[Theorem 7.8]{FN2}.
In this way, the results in the present paper are new and complement those in \cite{FN2}, as now we have a full description of the two Steinberg maps and a better understanding of them at the same time. Note that the results given in Section \ref{section-2.3} regarding $\Phi_\fk$ and for $p=q=r$ were already announced in \cite[\S 2]{FN3} without proofs.

\section{Main results}

\subsection{Combinatorial notation on pairs of partial permutations}

\label{section-2.1}

By $\ppermutationsof_{p,r}$ we denote the set of $p\times r$ matrices whose coefficients are $0$ or $1$, with at most one $1$ in each row and each column.
(If $p=r$, we recover the set of partial permutation matrices considered in \cite{FN2}.)
By $\ppermutationsof=\ppermutationsof_{(p,q),r}$ we denote the set of $(p+q)\times r$ matrices of rank $r$ (we have $r\leq p+q$) of the form
\[\omega=\begin{pmatrix} \tau_1 \\ \tau_2 \end{pmatrix},\]
where $\tau_1\in\ppermutationsof_{p,r}$ and $\tau_2\in\ppermutationsof_{q,r}$.
Note that the symmetric group $\permutationsof{r}$ acts on $\ppermutationsof$ by right multiplication, and we denote the quotient set by $\parameters=\ppermutationsof/\permutationsof{r}$.

{\em In Section \ref{section-2.2}, we will show that the elements of $\parameters$ parameterize the $K$-orbits of $\Xfv$.}

\medskip
\paragraph{\bf Graphic representation of a pair of partial permutations}\

We represent any element $\omega\in\parameters$ by a graph $\mathcal{G}(\omega)$ obtained as follows:
\begin{itemize}
\item
The set of vertices consists of $p$ ``positive'' vertices $1^+,\ldots,p^+$ and $q$ ``negative'' vertices $1^-,\ldots,q^-$, displayed along two horizontal lines.
\item Put an edge between $i^+$ and $j^-$ for every column of $\omega$ that contains exactly two $1$'s, in positions $i$ (within the block $\tau_1$) and $p+j$ (within the block $\tau_2$).
\item
Put a mark at the vertex $i^+$, respectively $j^-$, for every column of $\omega$ that contains exactly one $1$, in position $i$ (within $\tau_1$), respectively $p+j$ (within $\tau_2$).
\end{itemize}
For instance, for $(p,q)=(5,3)$ and $r=4$,
\begin{equation}
\label{2.1}
\omega=\mbox{\tiny $\begin{pmatrix} 0 & 0 & 0 & 0 \\ 1 & 0 & 0 & 0 \\ 0 & 0 & 0 & 0 \\ 0 & 1 & 0 & 0 \\ 0 & 0 & 1 & 0 \\ \hline 0 & 1 & 0 & 0 \\ 0 & 0 & 0 & 1 \\ 1 & 0 & 0 & 0 \end{pmatrix}$}
\qquad
\Rightarrow
\qquad
\mathcal{G}(\omega)=
\mbox{\tiny
$\begin{picture}(100,20)(0,0)
\put(0,11){$\bullet$}\put(20,11){$\bullet$}\put(40,11){$\bullet$}\put(60,11){$\bullet$}\put(80,11){$\bullet$}
\put(0,18){$1^+$}\put(20,18){$2^+$}\put(40,18){$3^+$}\put(60,18){$4^+$}\put(80,18){$5^+$}
\put(0,-11){$\bullet$}\put(20,-11){$\bullet$}\put(40,-11){$\bullet$}
\put(0,-20){$1^-$}\put(20,-20){$2^-$}\put(40,-20){$3^-$}
\put(21,13){\line(1,-1){21}}
\put(61,12){\line(-3,-1){60}}
\put(82,13){\circle{8}}
\put(22.5,-9){\circle{8}}
\end{picture}$}.
\end{equation}
We will say that a vertex is {\em free} whenever it is not a marked point nor an end point of an edge
(like $1^+$ and $3^+$ in the above example).

In general, the assignment $\omega\mapsto\mathcal{G}(\omega)$ establishes a bijection between $\parameters$ and the set of graphs with vertices 
$\{1^+,\ldots,p^+\}\cup\{1^-,\ldots,q^-\}$,
exactly $r$ edges or marked vertices, where every vertex is incident with at most one edge, and such that there is no edge which is incident with a marked vertex or joins two vertices of the same sign.

\medskip
\paragraph{\bf Numerical invariants for the graph}\

The following data associated to an element $\omega\in\parameters$ and its graphic representation $\mathcal{G}(\omega)$ will play a role in the statement of our main results.
\begin{itemize}
\item Set the {\em degree} of a vertex of $\mathcal{G}(\omega)$ as $0$, $1$, or $2$, depending on whether this vertex is free, incident with an edge, or marked.

We define $a^+(\omega)$, respectively $a^-(\omega)$, as the number of pairs
of positive vertices $(i^+,j^+)$ with $i<j$ and $\deg(i^+)<\deg(j^+)$,
respectively pairs of negative vertices $(i^-,j^-)$ with $i<j$ and $\deg(i^-)<\deg(j^-)$.

Let $b(\omega)$ be the number of edges of $\mathcal{G}(\omega)$.

Finally, let $c(\omega)$ be the number of \textit{crossings}, i.e., pairs of edges $(i^+,j^-)$, $(k^+,\ell^-)$ such that $i<k$ and $j>\ell$.

\item For all $(i,j)\in\{0,1,\ldots,p\}\times\{0,1,\ldots,q\}$, let $r_{i,j}(\omega)$ be the number of edges or marks contained in the subgraph of $\mathcal{G}(\omega)$ formed by the vertices $k^+$ ($1\leq k\leq i$), $\ell^-$ ($1\leq \ell\leq j$) and the edges/marks contained within this set of vertices. Let $R(\omega)=\left( r_{i,j}(\omega) \right)_{0\leq i\leq p,\,0\leq j\leq q}$ be the $(p+1)\times(q+1)$ matrix containing these numbers.

In particular, $r_{i,0}(\omega)$ (resp., $r_{0,j}(\omega)$) is the
number of marked vertices among
$\{1^+,\ldots,i^+\}$ (resp., $\{1^-,\ldots,j^-\}$).
\end{itemize}
{\em In Section \ref{section-2.2}, the numbers $a^\pm(\omega)$, $b(\omega)$, $c(\omega)$ appear in the dimension formula for the $K$-orbits of $\Xfv$, while the matrices $R(\omega)$ are used to describe the inclusion relations between orbit closures.}
\begin{itemize}
\item We decompose $\{1,\ldots,p\}=I\sqcup L\sqcup L'$ in the following way:
$I$, resp. $L$, resp. $L'$, denotes the set of elements $i\in\{1,\ldots,p\}$ such that $i^+$ is a vertex of $\mathcal{G}(\omega)$ of degree $1$, resp. $2$, resp. $0$.

We decompose $\{1,\ldots,q\}=J\sqcup M\sqcup M'$ in the same way: $J$, resp. $M$, resp. $M'$, consists of the elements $j$ such that $j^-$ has degree $1$, resp. $2$, resp. $0$.

Let $\sigma:J\to I$ be the bijection defined by letting $\sigma(j)=i$ if $(i^+,j^-)$ is an edge in $\mathcal{G}(\omega)$.
\end{itemize}
Note that $\omega$ is characterized by the subsets $I$, $L$, $L'$, $J$, $M$, $M'$ and the bijection $\sigma:J\to I$.
{\em In Section \ref{section-2.3},
these data are used to compute
the symmetrized and exotic Steinberg maps, by means of a combinatorial algorithm.}

Note also that we have $b(\omega)=\card{I}=\card{J}$, and $c(\omega)$ is the number of inversions of $\sigma$.

\begin{example}
\label{E2.1}
Let $\omega$ be as in (\ref{2.1}). Then,
\begin{eqnarray*}
 & a^+(\omega)=7,\quad a^-(\omega)=1,\quad b(\omega)=2,\quad c(\omega)=1, \quad
R(\omega)=\mbox{\tiny $\begin{pmatrix} 0 & 0 & 1 & 1 \\ 0 & 0 & 1 & 1 \\ 0 & 0 & 1 & 2 \\ 0 & 0 & 1 & 2 \\ 0 & 1 & 2 & 3 \\ 1 & 2 & 3 & 4 \end{pmatrix}$} \\
& 
\begin{array}{llll}
I=\{2,4\}, & L=\{5\}, & L'=\{1,3\},\\[2mm]
J=\{1,3\}, & M=\{2\}, & M'=\emptyset,
\quad\sigma=\begin{pmatrix} 1 & 3 \\ 4 & 2 \end{pmatrix}\in\mathrm{Bij}(J,I).
\end{array}
\end{eqnarray*}
Note that the matrix $R(\omega)$ can also be viewed as a plane partition.
\end{example}

\subsection{Orbit decomposition of the multiple flag variety $\Xfv$}

\label{section-2.2}

Recall that we consider the space $V=\C^{p+q}$ endowed with the polar decomposition
\[V=V^+\oplus V^-\qquad\mbox{where}\qquad V^+=\C^p\times\{0\}^q\quad\mbox{and}\quad V^-=\{0\}^p\times \C^q.\]
Let
\[
\flag_0^+=(\C^i\times\{0\}^{p-i}\times\{0\}^q)_{i=0}^p
\quad\mbox{and}\quad
\flag_0^-=(\{0\}^{p}\times\C^j\times\{0\}^{q-j})_{j=0}^q
\]
be the standard complete flags of $V^+$ and $V^-$.

Every  $(p+q)\times r$ matrix $\omega$ determines a subspace $[\omega]:=\mathrm{Im}\,\omega\subset V$, which remains the same up to permutation of the columns of $\omega$.
In particular, every $\omega\in\parameters$ determines a point $[\omega]$ in $\Grass(V,r)$, and thus a point $([\omega],\flag_0^+,\flag_0^-)$ in $\Xfv=\Grass(V,r)\times\Flags(V^+)\times\Flags(V^-)$.

\begin{theorem}
\label{T1}
\begin{enumerate}
\item Every $K$-orbit in $\Xfv$ is of the form $\Xorbit_\omega:=K\cdot([\omega],\flag_0^+,\flag_0^-)$ for a unique element $\omega\in\parameters=\ppermutationsof_{(p,q),r}/\permutationsof{r}$.
\item $\dim\Xorbit_\omega=\frac{p(p-1)}{2}+\frac{q(q-1)}{2}+a^+(\omega)+a^-(\omega)+\frac{b(\omega)(b(\omega)+1)}{2}+c(\omega)$.
\item $\Xorbit_\omega$ is the set of triples $(W,\flag^+=(F^+_i)_{i=0}^p,\flag^-=(F_j^-)_{j=0}^q)\in\Xfv$
satisfying the condition
\[\dim W\cap(F_i^++F_j^-)=r_{i,j}(\omega)\ \ \mbox{for all $(i,j)\in\{0,\ldots,p\}\times\{0,\ldots,q\}$}.\]
 \item $\overline{\Xorbit_\omega}\subset\overline{\Xorbit_{\omega'}}$ if and only if $r_{i,j}(\omega)\geq r_{i,j}(\omega')$ for all $(i,j)\in\{0,\ldots,p\}\times\{0,\ldots,q\}$.
\end{enumerate}
\end{theorem}

As a complement of this result, we determine the cover relations in the poset $(\{\overline{\Xorbit_\omega}\},\subset)$. We say that $\Xorbit_{\omega'}$ covers $\Xorbit_\omega$ if $\overline{\Xorbit_{\omega'}}$ strictly contains $\overline{\Xorbit_\omega}$ and is minimal (among the orbit closures) for this property.
Equivalently, this means that $\overline{\Xorbit_\omega}$ is an irreducible component of the boundary
$\partial\Xorbit_{\omega'}=\overline{\Xorbit_{\omega'}}\setminus \Xorbit_{\omega'}$.

\begin{theorem}
\label{C1}
The following conditions are equivalent:
\begin{enumerate}
\item $\Xorbit_{\omega'}$ covers $\Xorbit_\omega$;
\item $\dim\Xorbit_{\omega'}=\dim\Xorbit_\omega+1$ and (the graph of) $\omega$ is obtained from (the graph of) $\omega'$ by modifying the pattern of at most four vertices $a^+,b^+,c^-,d^-$ ($a<b$, $c<d$), according to one of the cases indicated in Figure \ref{figure1}.
\end{enumerate}

As a consequence, the boundary of every non-closed orbit is equidimensional of codimenison one.
\begin{figure}[h]
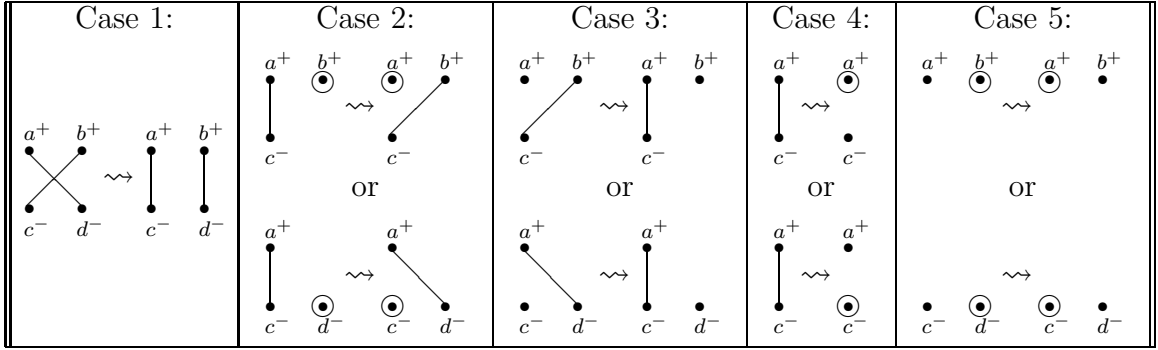

\begin{tabular}{||c|c|c|c|c||}
\hline {\rm Case 1:} & {\rm Case 2:} & {\rm Case 3:} & {\rm Case 4:} & {\rm Case 5:} \\
& & & &  \\[-2mm]
\coverA$\rightsquigarrow\ $\coverB\ &
\begin{tabular}{c}\coverC $\rightsquigarrow$ \coverD \\[6mm] {\rm or} \\[3mm]
\coverE$\rightsquigarrow\ $\coverF \end{tabular} &
\begin{tabular}{c}\coverG $\rightsquigarrow$ \coverH \\[6mm] {\rm or} \\[3mm]
\coverI$\rightsquigarrow\ $\coverJ \end{tabular} &
\begin{tabular}{c}\coverK $\rightsquigarrow$ \coverL \\[6mm] {\rm or} \\[3mm]
\coverM$\rightsquigarrow\ $\coverN \end{tabular} &
\begin{tabular}{c}\coverO $\rightsquigarrow$ \coverP \\[6mm] {\rm or} \\[3mm]
\coverQ$\rightsquigarrow\ $\coverR \end{tabular} \\[20mm] \hline
\end{tabular}
\caption{Elementary moves yielding cover relations in the poset $(\{\overline{\Xorbit_\omega}\},\subset)$.}
\label{figure1}
\end{figure}
\end{theorem}

It follows from Theorem \ref{C1} that the boundary $\partial\Xorbit_\omega:=\overline{\Xorbit_\omega}\setminus \Xorbit_\omega$ of every non-closed orbit is equidimensional of codimension one in $\overline{\Xorbit_\omega}$.
This boundary is in general not irreducible as it already appears in Example \ref{E2.4}\,{\rm (a)}.

\begin{example}
\label{E2.4}
(a)
In Figure \ref{figure2}, we represent the elements $\omega\in\parameters$ (under the form of their graphic incarnations $\mathcal{G}(\omega)$) in the case where $p=q=r=2$. We indicate the dimensions of the corresponding $K$-orbits $\Xorbit_\omega$. An edge joining two parameters indicates a cover relation.

\begin{figure}[h]
\textcolor{blue}{\xymatrixcolsep{1pc}
\xymatrix{
\textcolor{darkgray}{\dim:6}
& & & & & \graphA \ar@/_/@{-}[lld] \ar@{-}[d] \ar@/^/@{-}[rrd]  & & & & & \\
\textcolor{darkgray}{5} & & & \graphB \ar@{-}[lld] \ar@{-}[d] \ar@{-}[rrd] & & \graphC \ar@/_/@{-}[lllld] \ar@{-}[lld] \ar@{-}[rrd] \ar@/^/@{-}[rrrrd]  & & \graphD \ar@{-}[lld] \ar@{-}[d] \ar@{-}[rrd] & & & \\
\textcolor{darkgray}{4} & \graphE \ar@{-}[rd]\ar@/_/@{-}[rrrd] & & \graphF \ar@{-}[ld]\ar@/_/@{-}[rrrd] & & \graphG \ar@/_/@{-}[ld]\ar@/^/@{-}[rd] & & \graphH \ar@{-}[rd]\ar@/^/@{-}[llld] & & \graphI \ar@{-}[ld]\ar@/^/@{-}[llld] & \\
\textcolor{darkgray}{3} & & \graphJ \ar@{-}[rd]\ar@/_/@{-}[rrrd] & & \graphK \ar@{-}[rd] & & \graphL \ar@{-}[ld] & & \graphM \ar@{-}[ld]\ar@/^/@{-}[llld] & & \\
\textcolor{darkgray}{2} & & & \graphN & & \graphO & & \graphP & & &}}
\caption{The parameters of the $K$-orbits of $\Xfv$ and the cover relations for $p=q=r=2$.}
\label{figure2}
\end{figure}

(b) In Figure \ref{figure1}, the vertices $a^+,b^+$ or $c^-,d^-$ involved in an elementary move that yields a cover relation are not necessarily consecutive. For example, $\Xorbit_{\omega'}$ covers $\Xorbit_\omega$
if
\[
\mathcal{G}(\omega')=\graphnA
\qquad\qquad
\mathcal{G}(\omega)=\graphnB
\]

\bigskip
\noindent
This corresponds to Case 1 of Figure \ref{figure1} with $a=c=1$, $b=d=3$.
Note however that in Case 5 of Figure \ref{figure1}, the vertices $a^+,b^+$ (resp. $c^-,d^-$) must be consecutive for having a cover relation.
\end{example}

The proofs of Theorems \ref{T1} and \ref{C1} are given in Section \ref{section-3}.

One ingredient for showing the parametrization of the orbits in Theorem \ref{T1}\,(1) is that 
the orbits of a pair of Borel subgroups $B_p^+\times B_r^+\subset\GL_p(\C)\times\GL_r(\C)$ on the space of $p\times r$ matrices are parametrized by partial permutations
(Lemma \ref{L3.2}). 
This classification of orbits is also shown in \cite{Fulton-1991}, where dimension formulas, closure relations, and properties of closures of orbits are described.

\subsection{Description of symmetrized and exotic Steinberg maps}

\label{section-2.3}

We turn our attention to the maps $\Phi_\fk:\Xfv/K\to \nilpotentsof{\fk}/K$ and $\Phi_\fs:\Xfv/K\to\nilpotentsof{\fs}/K$ defined in Section \ref{section-1.2}.
The three orbit sets arising here can be parametrized combinatorially.
\begin{itemize}
\item $\Xfv/K=\{\Xorbit_\omega:\omega\in\parameters\}$ (see Section \ref{section-2.2}).
\end{itemize}
For the other two orbit sets, the parametrization is well known (see, e.g., \cite{Collingwood-McGovern}):
\begin{itemize}
\item $\nilpotentsof{\fk}$ is the nilpotent cone of the Lie algebra
$$\fk=\left\{\begin{pmatrix} a & 0 \\ 0 &
d\end{pmatrix}:a\in\gl_p(\C),\
d\in\gl_q(\C)\right\}\cong\gl_p(\C)\times\gl_q(\C),$$ 
and its adjoint
$K$-orbits $\Nkorbit_{(\lambda,\mu)}$ are parametrized by pairs of
partitions $\lambda\vdash p$ and $\mu\vdash q$ (viewed as Young
diagrams) through Jordan normal form. Specifically, the number of
boxes in the first $k$ columns of $\lambda$ (resp., $\mu$) indicates
the dimension of $\ker a^k$ (resp., $\ker d^k$).
\item $\nilpotentsof{\fs}$ is the nilpotent cone of
$$\fs=\left\{x=\begin{pmatrix} 0 & b \\ c &
0\end{pmatrix}:b\in\Mat_{p,q}(\C),\ c\in\Mat_{q,p}(\C)\right\},$$ 
and
its adjoint $K$-orbits $\Nsorbit_{\Lambda}$ are parametrized by
signed Young diagrams $\Lambda$ of signature $(p,q)$. Specifically,
the number of $+$'s (resp., $-$'s) in the first $k$ columns of
$\Lambda$ indicates the dimension of $V^+\cap\ker x^k$ (resp.,
$V^-\cap\ker x^k$) for $x\in\Nsorbit_\Lambda$.

\end{itemize}

We give a combinatorial algorithm which describes $\Phi_\fk$ and $\Phi_\fs$ completely. If $w:S\to R$ is a bijection
between two sets of integers, let $(\RS_1(w),\RS_2(w))$ denote the
pair of Young tableaux associated to $w$ via the Robinson--Schensted
correspondence, so that the set of entries of $\RS_1(w)$ (resp.,
$\RS_2(w)$) is $R$ (resp., $S$) (see, e.g., \cite{Fulton}).



Let $\omega\in\parameters$, and let $I,L,L',J,M,M',\sigma$ be the corresponding data in the sense of Section \ref{section-2.1}. Thus we have partitions $I\sqcup L\sqcup L'=\{1,\ldots,p\}$, $J\sqcup M\sqcup M'=\{1,\ldots,q\}$, and $\sigma:J\to I$ is a bijection.
We write $I=\{i_1<\ldots<i_k\}$, $J=\{j_1<\ldots<j_k\}$,
$L=\{\ell_1<\ldots<\ell_s\}$, $L'=\{\ell'_1<\ldots<\ell'_{s'}\}$, $M=\{m_1<\ldots<m_t\}$, $M'=\{m'_1<\ldots<m'_{t'}\}$,
and we consider the following permutations
\begin{eqnarray}
\label{wk+} & w_{\fk,+}=\begin{pmatrix} 1 & \cdots & s & s+1 & \cdots & s+k & s+k+1 & \cdots & p \\
\ell_s & \cdots & \ell_1 & \sigma(j_1) & \cdots & \sigma(j_k) & \ell'_{s'} & \cdots & \ell'_1\end{pmatrix}\in\permutationsof{p}, \\
\label{wk-} & w_{\fk,-}=\begin{pmatrix} 1 & \cdots & t & t+1 & \cdots & t+k & t+k+1 & \cdots & q \\
m_t & \cdots & m_1 & \sigma^{-1}(i_1) & \cdots & \sigma^{-1}(i_k) & m'_{t'} & \cdots & m'_1\end{pmatrix}\in\permutationsof{q}, 
\end{eqnarray}
and the bijections
\begin{eqnarray}
\label{ws+} & w_{\fs,+} = \begin{pmatrix} m_1 & \cdots & m_t & j_1 & \cdots & j_k & q+1 & \cdots & q+s' \\
-1 & \cdots & -t & \sigma(j_1) & \cdots & \sigma(j_k) & \ell'_{s'} & \cdots & \ell'_1 \end{pmatrix},
\end{eqnarray}
which maps $J\cup M\cup\{q+1,\ldots,q+s'\}$ to $\{-t,\ldots,-1\}\cup I\cup L'$, and
\begin{eqnarray}
\label{ws-} & w_{\fs,-} = \begin{pmatrix} \ell_1 & \cdots & \ell_s & i_1 & \cdots & i_k & p+1 & \cdots & p+t' \\
-1 & \cdots & -s & \sigma^{-1}(i_1) & \cdots & \sigma^{-1}(i_k) & m'_{t'} & \cdots & m'_1 \end{pmatrix},
\end{eqnarray}
which maps $I\cup L\cup\{p+1,\ldots,p+t'\}$ to $\{-s,\ldots,-1\}\cup J\cup M'$.

In the next theorem,
 $\nrboxes{\lambda}{c}$ denotes the number of boxes in the first $c$ columns of a Young diagram $\lambda$, and $\nrplus{\Lambda}{c}$ (resp., $\nrminus{\Lambda}{c}$) denotes the number of $+$'s (resp., $-$'s) in the first $c$ columns of a signed Young diagram $\Lambda$.

\begin{theorem}
\label{T2}
Let $\omega\in\parameters$, and consider the above notation.
\begin{enumerate}
\item
The image of $\Xorbit_\omega$ by the symmetrized Steinberg map is  $\Phi_\fk(\Xorbit_\omega)=\Nkorbit_{\lambda,\mu}$ where $(\lambda,\mu)$ is the pair of Young diagrams given by
\[
(\lambda,\mu)=\big(
\shape(\RS_1(w_{\fk,+})),\shape(\RS_1(w_{\fk,-}))\big).
\]
\item The image of $\Xorbit_\omega$ by the exotic Steinberg map is $\Phi_\fs(\Xorbit_\omega)=\Nsorbit_\Lambda$ where $\Lambda$ is the signed Young diagram determined as follows:
\begin{enumerate}
\item For every $c\geq 1$ even,
\[
\nrplus{\Lambda}{c}=\nrboxes{\lambda}{c}\quad\mbox{and}\quad \nrminus{\Lambda}{c}=\nrboxes{\mu}{c},
\]
where $(\lambda,\mu)$ is the pair of Young diagrams given in part {\rm (1)}.
\item For every $c\geq 1$ odd,
\[
\nrplus{\Lambda}{c}=s-t+\nrboxes{\lambda'}{c}\quad\mbox{and}\quad \nrminus{\Lambda}{c}=t-s+\nrboxes{\mu'}{c},
\]
where $(\lambda',\mu')$ is the pair of Young diagrams given by
\[(\lambda',\mu')=\big(\shape(\RS_1(w_{\fs,+})),\shape(\RS_1(w_{\fs,-}))\big).\]
\end{enumerate}
\end{enumerate}
\end{theorem}

We prove this theorem in Section \ref{section-4}.

\begin{example}
(a) For $\omega$ as in Example \ref{E2.1}, we have $s=t=1$,
\begin{eqnarray*}
 & w_{\fk,+}=\begin{pmatrix} 1 & 2 & 3 & 4 & 5 \\ 5 & 4 & 2 & 3 & 1 \end{pmatrix},\qquad w_{\fk,-}=\begin{pmatrix} 1 & 2 & 3 \\ 2 & 3 & 1 \end{pmatrix}, \\
 & w_{\fs,+}=\begin{pmatrix} 1 & 2 & 3 & 4 & 5 \\ 4 & -1 & 2 & 3 & 1 \end{pmatrix},\quad\mbox{and}\quad w_{\fs,-}=\begin{pmatrix} 2 & 4 & 5 \\ 3 & 1 & -1 \end{pmatrix}, \\
\end{eqnarray*}
hence we get
\[\lambda=\mbox{\scriptsize $\yng(2,1,1,1)$},\quad \mu=\mbox{\scriptsize $\yng(2,1)$},\quad\lambda'=\mbox{\scriptsize $\yng(3,1,1)$},\quad\mu'=\mbox{\scriptsize $\yng(1,1,1)$},\quad \mbox{and}\quad\Lambda=\mbox{\scriptsize $\young(-+,-+,+,+,+,-)$}.\]

(b) In Figure \ref{figure3}, we calculate the pair of Young diagrams $(\lambda,\mu)$ and the signed Young diagram $\Lambda$ such that $\Phi_\fk(\Xorbit_\omega)=\Nkorbit_{\lambda,\mu}$ and $\Phi_\fs(\Xorbit_\omega)=\Nsorbit_\Lambda$ for all $\omega\in\parameters$, for $p=q=r=2$ (the same case as in Example \ref{E2.4}).
\begin{figure}[h]
\begin{tabular}{|c|c|c|c|c|c|c|c|c|}
\hline
\begin{tabular}[b]{c} \ \\[0.1cm] $\mathcal{G}(\omega)$ \end{tabular} & \graphA & \ \ \graphB\ \  & \graphC & \graphD & \graphE & \graphF & \graphG & \graphH \\
\hline
& & & & & & & & \\[-4mm]
$\lambda,\mu$ & {\scriptsize $\yng(1,1),\yng(1,1)$} & {\scriptsize $\yng(1,1),\yng(1,1)$} & {\scriptsize $\yng(2),\yng(2)$}  & {\scriptsize $\yng(1,1),\yng(1,1)$} & {\scriptsize $\yng(1,1),\yng(2)$} & {\scriptsize $\yng(2),\yng(1,1)$} & {\scriptsize $\yng(1,1),\yng(1,1)$} & {\scriptsize $\yng(1,1),\yng(2)$} \\[4mm]
$\Lambda$ & {\scriptsize $\young(+,+,-,-)$} & {\scriptsize $\young(+-,+,-)$} & {\scriptsize $\young(+-,-+)$}  & {\scriptsize $\young(-+,+,-)$} & {\scriptsize $\young(+-,+,-)$} & {\scriptsize $\young(+-,+,-)$} & {\scriptsize $\young(+-,-+)$} & {\scriptsize $\young(-+,+,-)$} \\
\hline
\end{tabular}

\medskip

\begin{tabular}{|c|c|c|c|c|c|c|c|c|}
\hline
\begin{tabular}[b]{c} \ \\[0.1cm] $\mathcal{G}(\omega)$ \end{tabular} & \graphI & \graphJ & \graphK & \graphL & \graphM & \graphN & \graphO & \graphP \\
\hline
& & & & & & & & \\[-4mm]
$\lambda,\mu$ & {\scriptsize $\yng(2),\yng(1,1)$} & {\scriptsize $\yng(2),\yng(2)$} & {\scriptsize $\yng(1,1),\yng(2)$}  & {\scriptsize $\yng(2),\yng(1,1)$} & {\scriptsize $\yng(2),\yng(2)$} & {\scriptsize $\yng(1,1),\yng(1,1)$} & {\scriptsize $\yng(2),\yng(2)$} & {\scriptsize $\yng(1,1),\yng(1,1)$} \\[4mm]
$\Lambda$ & {\scriptsize $\young(-+,+,-)$} & {\scriptsize $\young(+-,+-)$} & {\scriptsize $\young(+-,-+)$}  & {\scriptsize $\young(+-,-+)$} & {\scriptsize $\young(-+,-+)$} & {\scriptsize $\young(+-,+-)$} & {\scriptsize $\young(+-,-+)$} & {\scriptsize $\young(-+,-+)$} \\[2mm]
\hline
\end{tabular}
\caption{Calculation of $\Phi_\fk(\Xorbit_\omega)=\Nkorbit_{\lambda,\mu}$ and $\Phi_\fs(\Xorbit_\omega)=\Nsorbit_\Lambda$ for $p=q=r=2$.}
\label{figure3}
\end{figure}
\end{example}

\begin{remark}
\label{R2.7}
Note that the pair of bijections $(w_{\fs,+},w_{\fs,-})$ of (\ref{ws+})--(\ref{ws-}) determines the original element $\omega\in\parameters$,
as the data $(I,J,L,L',M,M',\sigma)$ can be recovered from this pair.
On the contrary, the pair $(w_{\fk,+},w_{\fk,-})$ of (\ref{wk+})--(\ref{wk-}) does not determine $\omega$. For instance,
for the elements $\omega$ corresponding to the two graphs
\[
\graphC\qquad\mbox{and}\qquad\graphO
\]
we get the same pair of permutations $(w_{\fk,+},w_{\fk,-})=(\mathrm{id}_{\{1,2\}},\mathrm{id}_{\{1,2\}})$.
\end{remark}

The tableaux $\RS_1(w_{\fk,+})$ and $\RS_1(w_{\fk,-})$ involved in Theorem \ref{T2} can also be obtained as the result of the following combinatorial algorithms.
We need more notation:
\begin{itemize}
\item
If $T,S$ are Young tableaux with disjoint sets of entries, we denote by $T*S$ the rectification by jeu de taquin of the skew tableau obtained by displaying $S$ on the top right corner of $T$. For example,
\[
\young(13,6)*\young(245,7)=\mathrm{Rect}\left(\young(::245,::7,13,6)\right)=\young(1245,37,6).
\]
If $U$ is a third tableau whose entries do not appear in $T$ nor $S$, the properties of jeu de taquin imply that
$(T*S)*U=T*(S*U)$ (see \cite{Fulton}), hence the notation $T*S*U$ is unambiguous.
\item
Let $[L],[L'],[M],[M']$ denote the vertical Young tableaux whose entries are the elements in $L,L',M,M'$, respectively.
\end{itemize}
We then have:
\begin{equation}
\label{wk+star}
\RS_1(w_{\fk,+})=[L]*\RS_1(\sigma)*[L']
\end{equation}
and
\begin{equation}
\label{wk-star}
\RS_1(w_{\fk,-})=[M]*\RS_1(\sigma^{-1})*[M']=[M]*\RS_2(\sigma)*[M'].
\end{equation}

\begin{remark}
Assume that $p=q=r=n$ and $L=M'=\emptyset$, thus $s'=t=n-k$.
This special case is the one considered in \cite[\S9--10]{FN2} (except that the set $L$ in the notation of \cite[\S9--10]{FN2} corresponds to the set $L'$ in the notation of the present paper).
In this case:
\begin{enumerate}
\item The tableaux $\RS_1(w_{\fk,+})$ and $\RS_1(w_{\fk,-})$ coincide with the tableaux
$\RS_1(\sigma)*[L']$ and $[M]*\RS_2(\sigma)$
involved in \cite[Theorems 7.4, 9.1, and 10.4\,(1)]{FN2}.
\item The skew tableau obtained from $\RS_1(w_{\fs,+})$ by deleting the boxes with negative entries coincides with the skew tableau $[M]*\RS_2(\sigma)\bigtriangleup
\RS_1(\sigma)*[L']$ involved in \cite[Theorem 10.4\,(2)]{FN2}.
This follows from \cite[Lemma 10.9]{FN2}.
\item We have just $w_{\fs,-}=\sigma^{-1}$, hence $\RS_1(w_{\fs,-})=\RS_2(\sigma)$, which is the tableau involved in \cite[Theorem 10.4\,(3)]{FN2}.
\end{enumerate}
Thus, Theorem \ref{T2} recovers the results stated in \cite[Theorems 9.1 and 10.4]{FN2}.
\end{remark}

\subsection{An extension of the Robinson--Schensted correspondence}

\label{section-2.4}

As pointed out in Remark \ref{R2.7}, the pair of permutations $(w_{\fk,+},w_{\fk,-})$ of (\ref{wk+})--(\ref{wk-}), involved in the calculation of the symmetrized Steinberg map image $\Phi_\fk(\omega)$,
does not fully determine the element $\omega\in\parameters$.
A fortiori the map $\Phi_\fk$ itself is far from being injective.

In fact, we can determine the fibers of $\Phi_\fk$ in terms of a combinatorial correspondence which extends the Robinson--Schensted correspondence. The following theorem also generalizes \cite[Theorem 7.6]{FN2}.
We use the previous notation.
In addition, we write $\lambda'\columnstrip \lambda$ whenever $\lambda'$ is a Young subdiagram of $\lambda$ such that the skew diagram $\lambda\setminus\lambda'$ is {\em column strip} (i.e., it contains at most one box in each row).
Hereafter, $\partitionsof{n}$ denotes the set of partitions $\lambda\vdash n$, also seen as Young diagrams of size $|\lambda|=n$.

\begin{theorem}
\label{T3}
There is an explicit bijection 
\[
\mathrm{gRS}:\parameters\stackrel{\sim}{\longrightarrow}\mathcal{T}:=\bigsqcup_{(\lambda,\mu)\in\partitionsof{p}\times\partitionsof{q}} \mathcal{T}_{\lambda,\mu}
\]
where $\mathcal{T}_{\lambda,\mu}$ is the set of 5-tuples
$(T_1,T_2;\lambda',\mu';\nu)$ satisfying
\begin{itemize}
\item[$(\star)$] $T_1$ and $T_2$ are standard Young tableaux of shapes $\lambda$ and $\mu$, respectively;
\item[$(\star\star)$] $\nu\columnstrip\lambda'\columnstrip\lambda$,
$\nu\columnstrip\mu'\columnstrip\mu$, and $|\lambda'|+|\mu'|=|\nu|+r$.
\end{itemize}
Specifically, to the element $\omega\in\parameters$, we associate
the 5-tuple
\begin{eqnarray*}
\mathrm{gRS}(\omega)=(T_1,T_2;\lambda',\mu';\nu)&:=&\big([L]*\RS_1(\sigma)*[L'],[M]*\RS_2(\sigma)*[M']; \\
&&\shape([L]*\RS_1(\sigma)),\shape([M]*\RS_2(\sigma)); \\
&&\shape(\RS_1(\sigma))\big).
\end{eqnarray*}
\end{theorem}

By combining Theorems \ref{T2}\,(a) and \ref{T3}, we get a commutative diagram
\[
\xymatrix  @M 3pt @R 8ex{
\Xfv / K \simeq \parameters
\ar[drr]_{\Phi_\fk} \ar@{->}[rr]^-{\sim}_-{\genRS} & & \bigsqcup\limits_{(\lambda, \mu) \in \partitionsof{p} \times \partitionsof{q}}
\mathcal{T}_{\lambda,\mu}
\ar[d] &
\makebox[3ex][c]{\qquad$\ni (T_1, T_2; \lambda', \mu'; \nu)$} \ar@{|->}[d]\\
& & \partitionsof{p} \times \partitionsof{q} & \makebox[0ex][c]{$\ni \quad (\lambda, \mu) $\qquad} }
\]
from which we have that $\genRS$ restricts to a bijection 
$\Phi_\fk^{-1}(\lambda,\mu)\stackrel{\sim}{\to} \mathcal{T}_{\lambda,\mu}$.

\begin{proof}
First, we note that the considered map is well defined:
the fact that $\lambda\setminus\lambda'$, $\lambda'\setminus\nu$, $\mu\setminus\mu'$, and $\mu'\setminus\nu$ are column strips follows from \cite[Proposition in \S1.1]{Fulton},
and we have
\[
|\lambda'|+|\mu'|=(s+k)+(t+k)=k+(k+s+t)=|\nu|+r
\]
where, as before, $s=\card{L}$, $t=\card{M}$, $k=\card{I}=\card{J}$.

Next, we show that the map is bijective.
Let $(T_1,T_2;\lambda',\mu';\nu)\in\mathcal{T}_{\lambda,\mu}$.
Applying twice \cite[Proposition in \S1.1]{Fulton}, we find that there is a unique 6-tuple
$(S_1,S_2,L,L',M,M')$, where $S_1,S_2$ are Young tableaux of shape $\nu$
and $L,L',M,M'$ are sets of integers, such that $T_1=[L]*S_1*[L']$, $T_2=[M]*S_2*[M']$,
$\lambda'=\shape([L]*S_1)$, and $\mu'=\shape([M]*S_2)$
(understanding that the contents of $L,S_1,L'$ are disjoint, as well as those of $M,S_2,M'$).
Let $I$ (resp., $J$) be the set of entries of $S_1$ (resp., $S_2$).
By the Robinson--Schensted correspondence, we get $(S_1,S_2)=(\RS_1(\sigma),\RS_2(\sigma))$
for a unique bijection $\sigma:J\to I$. Then, the data
$I,L,L',J,M,M',\sigma$ determine a unique element $\omega\in\parameters$ (see Section \ref{section-2.1}) such that $\mathrm{gRS}(\omega)=(T_1,T_2;\lambda',\mu';\nu)$.
\end{proof}

\begin{example}
(a)
The 5-tuple corresponding to the element $\omega$ of Example \ref{E2.1} is
\[
\genRS(\omega)=(T_1,T_2;\lambda',\mu';\nu)=
\left(
{\tiny \ytableaushort{13,2,4,5}, \ytableaushort{13,2} ;
\ydiagram{1,1,1}, \ydiagram{2,1}; \ydiagram{1,1} }
\right)=
\left(
{\tiny \ytableaushort{13,2,4,5} *[*(red)]{1,1} *[*(yellow)]{1,1,1}, \ytableaushort{13,2} *[*(red)]{1,1} *[*(yellow)]{2,1} }
\right)
\]
(we encode the 5-tuple as the pair of tableaux $(T_1,T_2)$
where the boxes of $\lambda',\mu',\nu$ are colored).

(b)
In Figure \ref{newfigure}, we give the bijection of Theorem \ref{T3} in the case where $p=3$, $q=r=2$.
In this case, the set $\parameters$ has 34 elements.
\begin{figure}[t]
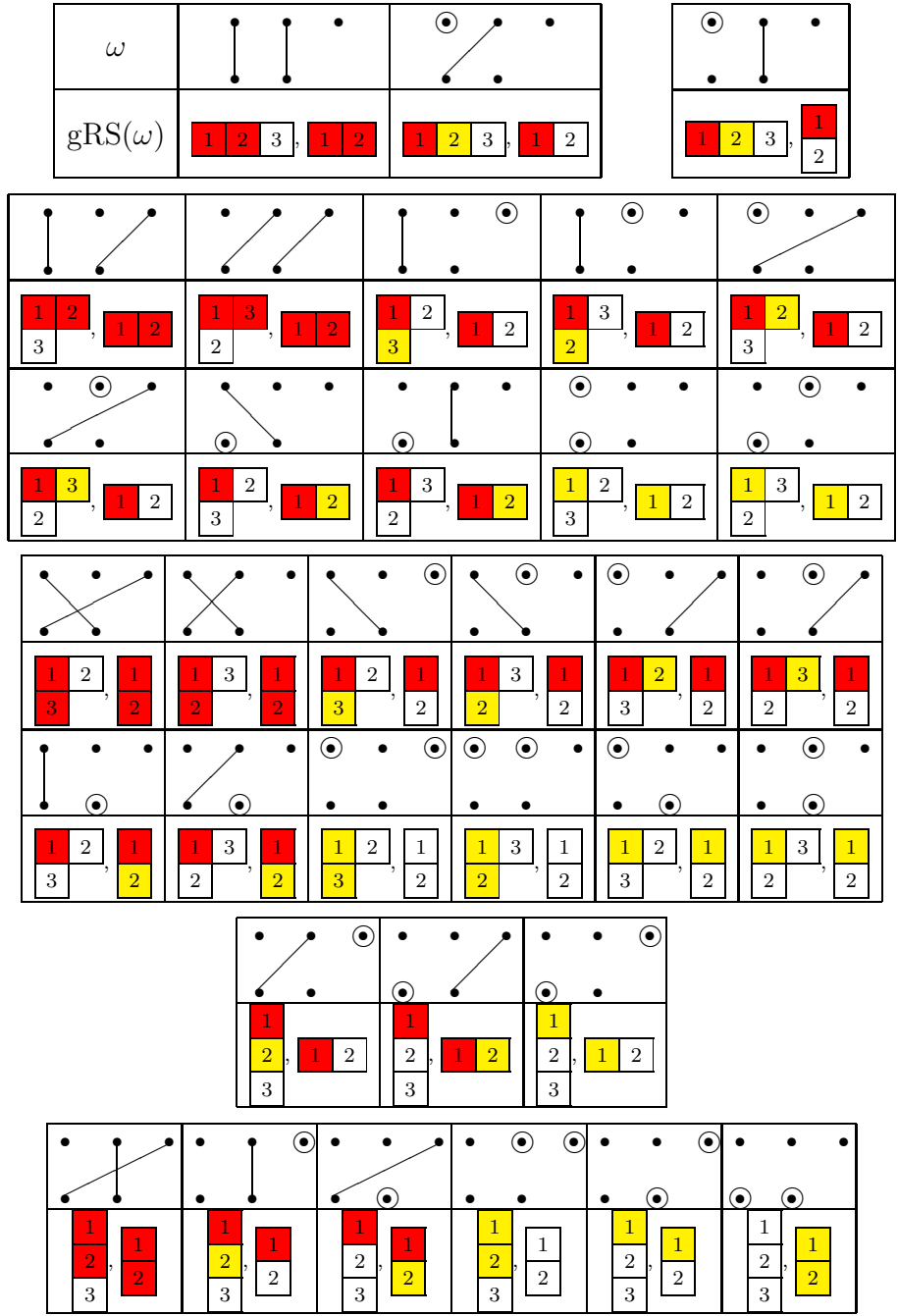

\begin{tabular}{|c|c|c|c|c|c|c|}
\hline
$\omega$ & \graphexA & \graphexK \\[3mm]
\hline 
&&\\[-1mm]
$\genRS(\omega)$ &
{\tiny \ytableaushort{123} *[*(red)]{2}, \ytableaushort{12} *[*(red)]{2} } &
{\tiny \ytableaushort{123} *[*(red)]{1} *[*(yellow)]{2}, \ytableaushort{12} *[*(red)]{1} }  \\[3mm]
\hline
\end{tabular}
\qquad
\begin{tabular}{|c|c|c|c|c|c|c|}
\hline
\graphexM \\[3mm] \hline 
\\[-3mm]
{\tiny \ytableaushort{123} *[*(red)]{1} *[*(yellow)]{2}, \ytableaushort{1,2} *[*(red)]{1} } \\[3mm]
\hline
\end{tabular}

\medskip

\begin{tabular}{|c|c|c|c|c|c|c|}
\hline
\graphexE & \graphexC & \graphexH & \graphexG & \graphexO \\[3mm]
\hline
&&&& \\[-3mm]
{\tiny \ytableaushort{12,3} *[*(red)]{2}, \ytableaushort{12} *[*(red)]{2} }
&
{\tiny \ytableaushort{13,2} *[*(red)]{2}, \ytableaushort{12} *[*(red)]{2} }
&
{\tiny \ytableaushort{12,3} *[*(red)]{1} *[*(yellow)]{1,1}, \ytableaushort{12} *[*(red)]{1} }
&
{\tiny \ytableaushort{13,2} *[*(red)]{1} *[*(yellow)]{1,1}, \ytableaushort{12} *[*(red)]{1} }
&
{\tiny \ytableaushort{12,3} *[*(red)]{1} *[*(yellow)]{2}, \ytableaushort{12} *[*(red)]{1} } \\[3mm]
\hline \cline{1-5}
\graphexP & \graphexT & \graphexU & \graphexBB & \graphexDD \\[3mm]
\cline{1-5} &&&& \\[-3mm]
{\tiny \ytableaushort{13,2} *[*(red)]{1} *[*(yellow)]{2}, \ytableaushort{12} *[*(red)]{1} }
&
{\tiny \ytableaushort{12,3} *[*(red)]{1}, \ytableaushort{12} *[*(red)]{1} *[*(yellow)]{2} }
&
{\tiny \ytableaushort{13,2} *[*(red)]{1}, \ytableaushort{12} *[*(red)]{1} *[*(yellow)]{2} }
&
{\tiny \ytableaushort{12,3} *[*(yellow)]{1}, \ytableaushort{12} *[*(yellow)]{1} }
&
{\tiny \ytableaushort{13,2} *[*(yellow)]{1}, \ytableaushort{12} *[*(yellow)]{1} }
\\[3mm]
\cline{1-5}
\end{tabular}

\medskip

\begin{tabular}{|c|c|c|c|c|c|c|}
\hline
\graphexF & \graphexB & \graphexJ & \graphexI & \graphexQ & \graphexR
\\[3mm]
\hline
&&&&& \\[-3mm]
{\tiny \ytableaushort{12,3} *[*(red)]{1,1}, \ytableaushort{1,2} *[*(red)]{1,1} }
&
{\tiny \ytableaushort{13,2} *[*(red)]{1,1}, \ytableaushort{1,2} *[*(red)]{1,1} }
&
{\tiny \ytableaushort{12,3} *[*(red)]{1} *[*(yellow)]{1,1}, \ytableaushort{1,2} *[*(red)]{1} }
&
{\tiny \ytableaushort{13,2} *[*(red)]{1} *[*(yellow)]{1,1}, \ytableaushort{1,2} *[*(red)]{1} }
&
{\tiny \ytableaushort{12,3} *[*(red)]{1} *[*(yellow)]{2}, \ytableaushort{1,2} *[*(red)]{1} }
&
{\tiny \ytableaushort{13,2} *[*(red)]{1} *[*(yellow)]{2}, \ytableaushort{1,2} *[*(red)]{1} } 
\\[3mm]
\hline
\cline{1-6}
\graphexS & \graphexV & \graphexZ & \graphexY & \graphexCC & \graphexEE \\[3mm]
\cline{1-6}
&&&&& \\[-3mm]
{\tiny \ytableaushort{12,3} *[*(red)]{1}, \ytableaushort{1,2} *[*(red)]{1} *[*(yellow)]{1,1} }
&
{\tiny \ytableaushort{13,2} *[*(red)]{1}, \ytableaushort{1,2} *[*(red)]{1} *[*(yellow)]{1,1} }
&
{\tiny \ytableaushort{12,3} *[*(yellow)]{1,1}, \ytableaushort{1,2} }
&
{\tiny \ytableaushort{13,2} *[*(yellow)]{1,1}, \ytableaushort{1,2} }
&
{\tiny \ytableaushort{12,3} *[*(yellow)]{1}, \ytableaushort{1,2} *[*(yellow)]{1} }
&
{\tiny \ytableaushort{13,2} *[*(yellow)]{1}, \ytableaushort{1,2} *[*(yellow)]{1} }
\\[3mm]
\cline{1-6}
\end{tabular}

\medskip

\begin{tabular}{|c|c|c|c|c|}
\hline
\graphexL & \graphexX & \graphexFF \\[3mm]
\hline
{\tiny \ytableaushort{1,2,3} *[*(red)]{1} *[*(yellow)]{1,1}, \ytableaushort{12} *[*(red)]{1} }
&
{\tiny \ytableaushort{1,2,3} *[*(red)]{1}, \ytableaushort{12} *[*(red)]{1} *[*(yellow)]{2} }
&
{\tiny \ytableaushort{1,2,3} *[*(yellow)]{1}, \ytableaushort{12} *[*(yellow)]{1} }
\\[4mm]
\hline
\end{tabular}

\medskip

\begin{tabular}{|c|c|c|c|c|c|}
\hline
\graphexD & \graphexN & \graphexW & \graphexAA & \graphexGG & \graphexHH  \\[3mm]
\hline
{\tiny \ytableaushort{1,2,3} *[*(red)]{1,1}, \ytableaushort{1,2} *[*(red)]{1,1} }
&
{\tiny \ytableaushort{1,2,3} *[*(red)]{1} *[*(yellow)]{1,1}, \ytableaushort{1,2} *[*(red)]{1} }
&
{\tiny \ytableaushort{1,2,3} *[*(red)]{1}, \ytableaushort{1,2} *[*(red)]{1} *[*(yellow)]{1,1} }
&
{\tiny \ytableaushort{1,2,3} *[*(yellow)]{1,1}, \ytableaushort{1,2} }
&
{\tiny \ytableaushort{1,2,3} *[*(yellow)]{1}, \ytableaushort{1,2} *[*(yellow)]{1} }
&
{\tiny \ytableaushort{1,2,3}, \ytableaushort{1,2} *[*(yellow)]{1,1} }
\\[4mm]
\hline
\end{tabular}
\caption{The correspondence $\omega\mapsto(T_1,T_2;\lambda',\mu';\nu)$ in the case $(p,q,r)=(3,2,2)$.}
\label{newfigure}
\end{figure}
\end{example}

\begin{remark}
We point out that Singh \cite{Singh} has recently developed a Robinson--Schensted correspondence for partial permutations.
Specifically, he has established a bijection between the set of partial permutations of size $p\times q$
and a set of triples $(\Lambda,P,Q)$ consisting of a signed Young diagram of size $p+q$ and two standard Young tableaux of sizes $p$ and $q$.
Note that, if $r=q$, the set of partial permutations of size $p\times q$ can be realized in a natural way as a subset of our set $\parameters$. One can ask whether there is a relation between the correspondence in Theorem \ref{T3} and
the bijection given in \cite[Theorem A]{Singh}.

We refer to \cite{Rosso,Trapa-IMRN,Travkin} for other generalizations of the Robinson--Schensted correspondence
that arise in geometry.
\end{remark}


We derive from Theorem \ref{T3} an interpretation of the cardinals of the fibers $\Phi_\fk^{-1}(\Nkorbit_{\lambda,\mu})$
based on representation theory.
If $\lambda\in\partitionsof{n}$, let $\rho_\lambda^{(n)}$ denote the corresponding irreducible representation of $\permutationsof{n}$.
In this way, the (isomorphism classes of) irreducible representations
$\rho_\lambda^{(p)}\boxtimes \rho_\mu^{(q)}$
of $\permutationsof{p}\times\permutationsof{q}$ are parametrized by
the pairs of partitions $(\lambda,\mu)\in\partitionsof{p}\times\partitionsof{q}$.
Here $\boxtimes$ stands for the outer tensor product.

For every triple of nonnegative integers $(k,s,t)$ such that
\begin{equation}
\label{kst}
s':=p-k-s\geq 0,\quad t':=q-k-t\geq 0,\quad k+s+t=r,
\end{equation}
we consider the subgroup of $\permutationsof{p}\times \permutationsof{q}$ given by
\begin{eqnarray*}
H_{k,s,t} & = & \{(a_1,a_2,a_3;b_1,b_2,b_3)\in(\permutationsof{k}\times\permutationsof{s}\times \permutationsof{s'})\times
(\permutationsof{k}\times\permutationsof{t}\times \permutationsof{t'}):a_1=b_1\} \\
 & \cong & \Delta\permutationsof{k}\times \permutationsof{s}\times \permutationsof{s'}\times \permutationsof{t}\times \permutationsof{t'},
\end{eqnarray*}
where $\Delta\permutationsof{k}$ stands for the diagonal embedding of $\permutationsof{k}$ in $\permutationsof{k}\times\permutationsof{k}$.
Let $\varepsilon$ denote the signature representation of $H_{k,s,t}$ (the restriction of $\rho_{(1^p)}^{(p)}\boxtimes\rho_{(1^q)}^{(q)}$).
The induced representation $\mathrm{Ind}_{H_{k,s,t}}^{\permutationsof{p}\times\permutationsof{q}}\varepsilon$ decomposes as
a  sum of irreducible representations
\[
\mathrm{Ind}_{H_{k,s,t}}^{\permutationsof{p}\times\permutationsof{q}}\varepsilon=\bigoplus_{(\lambda,\mu)\in\partitionsof{p}\times\partitionsof{q}} (\rho_\lambda^{(p)}\boxtimes \rho_\mu^{(q)})^{m_{k,s,t}(\lambda,\mu)},
\]
where $m_{k,s,t}(\lambda,\mu)$ denotes the multiplicity of $\rho_\lambda^{(p)}\boxtimes \rho_\mu^{(q)}$.
The next corollary is a generalization of \cite[Corollary 7.10]{FN2}.

\begin{corollary}
For every pair of partitions $(\lambda,\mu)\in\partitionsof{p}\times\partitionsof{q}$, we have
\[\card\Phi_\fk^{-1}(\Nkorbit_{\lambda,\mu})=\sum_{(k,s,t)}m_{k,s,t}(\lambda,\mu)\dim \rho_\lambda^{(p)}\boxtimes \rho_\mu^{(q)},\]
where the sum is over triples $(k,s,t)$ satisfying (\ref{kst}).
\end{corollary}

\begin{proof}
Note that $\varepsilon=\mathbf{1}\boxtimes\varepsilon'$, where $\mathbf{1}$ is the trivial representation of $\Delta\permutationsof{k}$ and $\varepsilon'$ is the signature representation of $\permutationsof{s}\times\permutationsof{s'}\times\permutationsof{t}\times\permutationsof{t'}$.
Let
\[
\tilde{H}_{k,s,t}=(\permutationsof{k}\times\permutationsof{s}\times \permutationsof{s'})\times
(\permutationsof{k}\times\permutationsof{t}\times \permutationsof{t'}).
\]
The intermediate induced representation $\mathrm{Ind}_{H_{k,s,t}}^{\tilde{H}_{k,s,t}}\varepsilon$ can be written as
\begin{eqnarray*}
\mathrm{Ind}_{H_{k,s,t}}^{\tilde{H}_{k,s,t}}\varepsilon
 & = & (\mathrm{Ind}_{\Delta\permutationsof{k}}^{\permutationsof{k}\times\permutationsof{k}}\mathbf{1})\boxtimes\varepsilon'
=\mathbb{C}\permutationsof{k}\boxtimes\rho_{(1^s)}^{(s)}\boxtimes\rho_{(1^{s'})}^{(s')}\boxtimes\rho_{(1^t)}^{(t)}\boxtimes\rho_{(1^{t'})}^{(t')} \\
 & = & \bigoplus_{\nu\in\partitionsof{k}}(\rho_\nu^{(k)}\boxtimes\rho_{(1^s)}^{(s)}\boxtimes\rho_{(1^{s'})}^{(s')})\boxtimes(\rho_\nu^{(k)}\boxtimes\rho_{(1^t)}^{(t)}\boxtimes\rho_{(1^{t'})}^{(t')})^*,
\end{eqnarray*}
where $\mathbb{C}\permutationsof{k}=\bigoplus_{\nu\in\partitionsof{k}}\rho_\nu^{(k)}\boxtimes(\rho_\nu^{(k)})^*$ is the regular representation of $\permutationsof{k}\times\permutationsof{k}$.
Applying twice the Pieri rule (see \cite[\S2.2 and \S7.3]{Fulton}), we have
\[\mathrm{Ind}_{\permutationsof{k}\times\permutationsof{s}\times\permutationsof{s'}}^{\permutationsof{p}}(\rho_\nu^{(k)}\boxtimes\rho_{(1^s)}^{(s)}\boxtimes\rho_{(1^{s'})}^{(s')})=
\bigoplus_{\substack{\lambda'\in\partitionsof{k+s}\\ \mathrm{with}\,\nu\smallcolumnstrip\lambda'}}\bigoplus_{\substack{\lambda\in\partitionsof{p}\\\mathrm{with}\,\lambda'\smallcolumnstrip\lambda}}\rho_\lambda^{(p)}.\]
Since $\mathrm{Ind}_{H_{k,s,t}}^{\permutationsof{p}\times\permutationsof{q}}\varepsilon=
\mathrm{Ind}_{\tilde{H}_{k,s,t}}^{\permutationsof{p}\times\permutationsof{q}}
\mathrm{Ind}_{H_{k,s,t}}^{\tilde{H}_{k,s,t}}\varepsilon$, we deduce that
\[
m_{k,s,t}(\lambda,\mu)=\card\{(\nu,\lambda',\mu')\in\partitionsof{k}\times\partitionsof{k+s}\times\partitionsof{k+t}:\nu\columnstrip\lambda'\columnstrip\lambda,\ \nu\columnstrip\mu'\columnstrip\mu\}
\]
for all triples $(k,s,t)$.
Note also that $\dim\rho_\lambda^{(p)}\boxtimes\rho_\mu^{(q)}$ is the number of pairs of standard Young tableaux $(T_1,T_2)$ of shape $\lambda$ and $\mu$, respectively.
The claimed equality now follows from Theorem \ref{T3}.
\end{proof}

\begin{corollary}
The total number of $K$-orbits in $\Xfv$ is given by
\[\card\Xfv/K=\sum_{(k,s,t)}\dim \mathrm{Ind}_{H_{k,s,t}}^{\permutationsof{p}\times\permutationsof{q}}\varepsilon=\sum_{(k,s,t)}\binom{p}{k,s,s'}\binom{q}{k,t,t'}k!,\]
where the sums are over triples $(k,s,t)$ satisfying (\ref{kst}).
\end{corollary}

\section{On the decomposition of $\Xfv$ into $K$-orbits}

\label{section-3}

The purpose of this section is to prove the results stated in Theorems \ref{T1} and \ref{C1}, regarding the decomposition of $\Xfv=\Grass(V,r)\times\Flags(V^+)\times\Flags(V^-)$ into orbits of $K=\GL(V^+)\times\GL(V^-)$.

By $B_k^+\subset\GL_k(\C)$ we denote the subgroup of invertible upper triangular matrices in $\GL_k(\C)$.
Then\[B_K=\left\{\begin{pmatrix}
a & 0 \\ 0 & d
\end{pmatrix}:a\in B_p^+,\ d\in B_q^+\right\}=B_p^+\times B_q^+\subset\GL(V^+)\times\GL(V^-)\]
is a Borel subgroup of $K$. Recall that $\flag_0^+$ and $\flag_0^-$ denote the standard flags of $V^+$ and $V^-$. Thus, $B_K$ is the stabilizer of the pair $(\flag_0^+,\flag_0^-)$ for the action of $K$ on the product of flag varieties $\Flags(V^+)\times\Flags(V^-)$.
In Section \ref{section-3.1}, we observe that there is a one-to-one correspondence between the orbit sets $\Xfv/K$ and $\Grass(V,r)/B_K$, which preserves the inclusion relations between closures of orbits. This fact is a useful ingredient in the rest of the section.

In Section \ref{section-3.2}, we show the parametrization of orbits and the dimension formula stated in Theorem \ref{T1}\,(1)--(3). In Section \ref{section-3.3}, we describe the closure relations of orbits
by proving Theorems \ref{T1}\,(4) and \ref{C1}.
In Section \ref{section-3.4}, we make further remarks and mention relations with the existing literature.

\subsection{A preliminary lemma}

\label{section-3.1}

We will use the following lemma.

\begin{lemma}
\label{L3.1}
\begin{enumerate}
\item The mapping $\Grass(V,r)\to\Xfv$, $W\mapsto (W,\flag_0^+,\flag_0^-)$ induces a one-to-one correspondence between the orbit sets
\[\Xi:\Grass(V,r)/B_K\to \Xfv/K,\ \Gorbit=B_K\cdot W\mapsto \Xi(\Gorbit)=K\cdot (W,\flag_0^+,\flag_0^-).\]
\item If $\Xorbit=\Xi(\Gorbit)\subset\Xfv$ is the $K$-orbit corresponding to $\Gorbit\subset\Grass(V,r)$, then $\Xorbit\cong K\times^{B_K}\Gorbit$. In particular, $\dim\Xorbit=\dim\Gorbit+\dim K/B_K$.
\item The correspondence preserves the closure relations. Namely, if $\Gorbit_1,\Gorbit_2$ are $B_K$-orbits of $\Grass(V,r)$ and if $\Xorbit_1=\Xi(\Gorbit_1)$ and $\Xorbit_2=\Xi(\Gorbit_2)$ are the corresponding $K$-orbits of $\Xfv$, then
\[\overline{\Gorbit_1}\subset\overline{\Gorbit_2}\iff \overline{\Xorbit_1}\subset\overline{\Xorbit_2}.\]
\end{enumerate}
\end{lemma}

Lemma \ref{L3.1} is a consequence of the following lemma (which applies to a general connected reductive group $K$).
As already used in Lemma \ref{L3.1}, $K\times^QX$ stands for the quotient of $K\times X$ by the action of $Q$ given by $q\cdot(k,x)=(kq^{-1},qx)$, and $[k,x]$ denotes the class of $(k,x)$ in this quotient. 

\begin{lemma}
\label{L3.2-new}
Let $K$ be a connected reductive group and let $Q\subset K$ be a parabolic subgroup. Let $X$ be an algebraic variety endowed with an action of $K$. Consider the diagonal action of $K$ on $\mathbb{X}:=K/Q\times X$.
Note that there is an isomorphism  $\chi:\mathbb{X}\to K\times^QX$ given by $\chi(kQ,x)=[k,k^{-1}x]$.
Also we consider the closed immersion $\iota:X\to\mathbb{X}$, $x\mapsto(Q,x)$.
\begin{enumerate}
\item There is a one-to-one correspondence (in fact an isomorphism of partially ordered sets)
\[\Xi:\{\mbox{$Q$-stable subsets $M\subset X$}\}\to\{\mbox{$K$-stable subsets $N\subset\mathbb{X}$}\}\]
given by $\Xi(M)=K\cdot\iota(M)$. The inverse bijection is given by $N\mapsto\iota^{-1}(N)$. Moreover, $\Xi$ restricts to a one-to-one correspondence between the orbit sets $X/Q$ and $\mathbb{X}/K$.
\item Every $Q$-stable subset $M\subset X$ yields a subset $K\times^QM\subset K\times^QX$, and we have $\chi(\Xi(M))=K\times^QM$.
\item Let $N=\Xi(M)$ for some $Q$-stable subset $M$. Then, $M$ is closed in $X$ if and only if $N$ is closed in $\mathbb{X}$. More generally, we have $\overline{N}=\Xi(\overline{M})$.
\end{enumerate}
\end{lemma}

Though this lemma is well known, we give a proof for the sake of completeness.

\begin{proof}
(1) The map $\Xi$ is certainly well defined.
For a $Q$-stable subset $M\subset X$,
the inclusion $M\subset\iota^{-1}(K\cdot\iota(M))=\iota^{-1}(\Xi(M))$ is clear, while for $x\in\iota^{-1}(\Xi(M))$ there are $k\in K$ and $y\in M$ such that $\iota(x)=k\cdot\iota(y)$, which means that $(Q,x)=(kQ,ky)$, whence $k\in Q$ and $x=ky\in Q\cdot M=M$. We have shown that $M=\iota^{-1}(\Xi(M))$.

For a $K$-stable subset $N\subset \mathbb{X}$,
given $x\in \iota^{-1}(N)$ and $q\in Q$ we have
\[\iota(qx)=(Q,qx)=q\cdot (Q,x)\in N,\]
hence $qx\in\iota^{-1}(N)$; this shows that $\iota^{-1}(N)$ is $Q$-stable.

Since $N$ is $K$-stable, the inclusion $\Xi(\iota^{-1}(N))=K\cdot \iota(\iota^{-1}(N))\subset N$ is clear, while for $(kQ,x)\in N$ we have
$\iota(k^{-1}x)=(Q,k^{-1}x)=k^{-1}\cdot(kQ,x)\in N$, hence $(kQ,x)\in K\cdot\iota(\iota^{-1}(N))=\Xi(\iota^{-1}(N))$. This shows that $N=\Xi(\iota^{-1}(N))$.

We have thus shown that $\Xi$ is a bijection, with inverse bijection given by $N\mapsto \iota^{-1}(N)$. Note also that the implications
\[(M\subset M'\Rightarrow \Xi(M)\subset\Xi(M'))\quad\mbox{and}\quad (N\subset N'\Rightarrow \iota^{-1}(N)\subset\iota^{-1}(N'))\]
are clear, which show that $\Xi$ is in fact an isomorphism of posets.

Finally, if $M=Q\cdot x$ is a $Q$-orbit, then $\Xi(M)=K\cdot(Q\cdot\iota(x))=K\cdot \iota(x)$ is a $K$-orbit. If $N=K\cdot(kQ,x)$ is a $K$-orbit, then $N=K\cdot(Q,k^{-1}x)=\Xi(Q\cdot(k^{-1}x))$ is the image of a $Q$-orbit. The proof of part (1) is complete.

(2) First we note that, if $M\subset X$ is $Q$-stable, then the quotient $K\times^QM=(K\times M)/Q$ coincides with the subset $\{[k,x]\in K\times^QX:x\in M\}\subset K\times^QX$. This subset can also be written as
\[K\times^QM=\{\chi(kQ,kx):k\in K,\ x\in M\}=\chi(K\cdot\iota(M))=\chi(\Xi(M)).\]

(3) Let $N=\Xi(M)$. If $N$ is closed, then $M=\iota^{-1}(N)$ is closed.
Conversely, assume that $M$ is closed.
Then, $\iota(M)\subset \mathbb{X}$ is closed and $Q$-stable, and this implies that
the set $\{(k,\xi)\in K\times\mathbb{X}:k^{-1}\cdot \xi\in\iota(M)\}$ is closed as well as its image in $K/Q\times\mathbb{X}$. Note that
\[N=K\cdot \iota(M)=\mathrm{pr}_2(\{(kQ,\xi)\in K/Q\times\mathbb{X}:k^{-1}\cdot \xi\in\iota(M)\}).\]
Since $K/Q$ is complete, we conclude that $N$ is closed.


More generally, using that $\Xi(\overline{M})$ and $\Xi^{-1}(\overline{N})$ are closed, and the fact that $\Xi$ is an isomorphism of posets, we get
\[\overline{N}=\overline{\Xi(M)}\subset\Xi(\overline{M})=\Xi(\overline{\Xi^{-1}(N)})\subset\Xi(\Xi^{-1}(\overline{N}))=\overline{N},\]
hence $\overline{N}=\Xi(\overline{M})$, as claimed.\end{proof}

\begin{remark}
In Lemma \ref{L3.2-new}, the assumption that $Q$ is parabolic is used only in part (3).
Also note that the isomorphism $\chi$ is $K$-equivariant when $K$ acts on $\mathbb{X}=K/Q\times X$ diagonally and on $K\times^Q X$ by left multiplication.
\end{remark}

\subsection{Parametrization and dimension formula -- proof of  Theorem \ref{T1}\,(1)--(3)}\label{section-3.2}

As before, we denote by $B_p^+\subset\GL_p(\C)$ (resp., $B_r^+\subset\GL_r(\C)$) the Borel subgroup of upper triangular matrices. We need two lemmas. The first one is an analogue of \cite[Proposition 6.3]{FN2}.
It is also shown in \cite[p.~390]{Fulton-1991}, but we give a proof for the sake of completeness.

\begin{lemma}
\label{L3.2}
Every $p\times r$ matrix can be written in the form
$b_1\tau b_2$
for some $b_1\in B_p^+$, $b_2\in B_r^+$, and a unique $\tau\in\ppermutationsof_{p,r}$.
\end{lemma}

\begin{proof}
Through Gauss elimination, any $p\times r$ matrix $a$ can be
transformed into a matrix $\tau\in\ppermutationsof_{p,r}$ by a
series of operations consisting of multiplying a row (resp., a
column) by a nonzero scalar or adding to a row (resp., to a column)
another row (resp., column) situated below it (resp., on its left).
These operations correspond to multiplying on the left (resp., on
the right) by an element of $B_p^+$ (resp., $B_r^+$). Hence the
double coset $B_p^+aB_r^+$ contains an element
$\tau\in\ppermutationsof_{p,r}$, which means that $a\in
B_p^+ \tau B_r^+$.

For every pair $(i,j)\in\{1,\ldots,p\}\times\{1,\ldots,r\}$, the mapping $a\mapsto \beta_{i,j}(a):=\rank\,(a_{k,\ell})_{\substack{i\leq k\leq p \\ 1\leq \ell\leq j}}$
is constant on the set $B_p^+\tau B_r^+$, and we have
\[\beta_{i,j}(a)=\beta_{i,j}(\tau)=\card\{\mbox{$1$'s within the submatrix $(\tau_{k,\ell})_{\substack{i\leq k\leq p \\ 1\leq \ell\leq j}}$}\}.\]
This implies that two different elements $\tau,\tau'\in\ppermutationsof_{p,r}$ cannot belong to the same double coset $B_p^+aB_r^+$, whence the uniqueness.
\end{proof}

The second lemma is analogous to \cite[Lemma 8.2]{FN2}.

\begin{lemma}
\label{L3.3}
For every $\tau\in\ppermutationsof_{p,r}$, there is a permutation $w\in\permutationsof{r}$
such that $\tau w B_r^+\subset B_p^+\tau w$.
\end{lemma}

\begin{proof}
By $(e_1^\ell,\ldots,e_\ell^\ell)$ we denote the standard basis of $\C^\ell$. By $e_{i,j}^\ell$ we denote the elementary $\ell\times \ell$ matrix
whose $(i,j)$ coefficient is $1$ and the other coefficients are $0$. By $\mathrm{diag}(t_1,\ldots,t_\ell)$, we denote the diagonal matrix with coefficients $t_1,\ldots,t_\ell$ along the diagonal.

Let $k=\rank\,\tau$. Let $1\leq i_1<\ldots<i_k\leq p$ be such that $\mathrm{Im}\,\tau=\Span{e_{i_1}^p,\ldots,e_{i_k}^p}$.
We choose $w\in\permutationsof{r}$ such that $\tau w(e_j^r)=0$ for $1\leq j\leq r-k$ and $\tau w(e_{r-k+j}^r)=e_{i_j}^p$
for $1\leq j\leq k$.

For every diagonal matrix $t=\diag(t_1,\ldots,t_r)\in B_r^+$, we have
\[\tau w t=t'\tau w\in B_p^+\tau w\]
for any diagonal matrix $t'=\diag(t'_1,\ldots,t'_p)\in B_p^+$ such that $t'_{i_j}=t_{r-k+j}$ for all $j\in\{1,\ldots,k\}$.
For every transvection $u=1_r+xe_{j,\ell}^r\in B_r^+$ where $1\leq j<\ell\leq r$, we have
\[
\tau w u=\left\{\begin{array}{ll}
\tau w & \mbox{if $j\leq r-k$}, \\
(1_p+xe_{i_{j-(r-k)},i_{\ell-(r-k)}}^p)\tau w & \mbox{if $j>r-k$},
\end{array}\right.
\]
hence $\tau wu\in B_p^+\tau w$ in each case. Since $B_r^+$ is generated by the diagonal matrices and the transvections, we conclude that $\tau w B_r^+\subset B_p^+\tau w$.
\end{proof}

Now we are ready to prove parts (1) and (3) of Theorem \ref{T1}.

\begin{proof}[Proof of Theorem \ref{T1}\,(1) and (3)]
Every $W\in\Grass(V,r)$ is the image of a matrix
\[a=\begin{pmatrix} a_1 \\ a_2 \end{pmatrix}\quad\mbox{with}\quad a_1\in\Mat_{p,r}(\C),\ a_2\in\Mat_{q,r}(\C),\ \rank\,a=r.\]
By Lemma \ref{L3.2}, there are $\tau_1\in \ppermutationsof_{p,r}$,
$b_1\in B_p^+$, $b_2\in B_r^+$ such that $a_1=b_1\tau_1b_2$. Moreover, by Lemma \ref{L3.3}, there is $w\in\permutationsof{r}$ such that $\tau_1wB_r^+\subset B_p^+\tau_1w$. We have
\[
a=\begin{pmatrix} a_1 \\ a_2 \end{pmatrix}= \begin{pmatrix} b_1\tau_1w \\ a'_2\end{pmatrix}w^{-1}b_2
\]
for some $a'_2\in\Mat_{q,r}(\C)$. Applying again Lemma \ref{L3.2}, there are $\tau_2\in\ppermutationsof_{q,r}$, $b_3\in B_r^+$, and $b_4\in B_q^+$ such that $a'_2=b_4\tau_2b_3$. Moreover, there is $b'_1\in B_p^+$ such that $\tau_1wb_3^{-1}=b'_1\tau_1w$. This yields
\[a=\begin{pmatrix} b_1b'_1\tau_1w \\ b_4\tau_2\end{pmatrix}b_3w^{-1}b_2=\begin{pmatrix} b_1b'_1 & 0 \\ 0 & b_4\end{pmatrix}\begin{pmatrix}\tau_1w \\ \tau_2\end{pmatrix}b_3w^{-1}b_2\]
hence
\[
W=\mathrm{Im}\,a\in B_K\cdot[\omega]\quad\mbox{where}\quad \omega=\begin{pmatrix} \tau_1w \\ \tau_2\end{pmatrix}\in\ppermutationsof_{(p,q),r}.
\]
This implies that $\Grass(V,r)=\bigcup_{\omega\in\parameters}B_K\cdot[\omega]$, hence $\Xfv=\bigcup_{\omega\in\parameters}\Xorbit_\omega$ in view of Lemma \ref{L3.1}.

The mappings
\[\xi=(W,(F^+_i)_{i=0}^p,(F^-_j)_{j=0}^q)\mapsto d_{i,j}(\xi):=\dim W\cap (F_i^++F_j^-),\] for $(i,j)\in\{0,\ldots,p\}\times\{0,\ldots,q\}$, are constant on every $K$-orbit of $\Xfv$.
Let $(e_1^+,\ldots,e_p^+)$ (resp., $(e_1^-,\ldots,e_q^-$)) be the standard basis of $V^+=\C^p\times\{0\}^q$ (resp., $V^-=\{0\}^p\times\C^q$), so that the standard flags $\flag_0^\pm$ are given by $\flag_0^+=(\Span{e_1^+,\ldots,e_i^+})_{i=0}^p$ and $\flag_0^-=(\Span{e_1^-,\ldots,e_j^-})_{j=0}^q$.
For every $\omega\in\parameters$,
the definition of the graph $\mathcal{G}(\omega)$ implies that the subspace \[[\omega]\cap(\Span{e_1^+,\ldots,e_i^+}+\Span{e_1^-,\ldots,e_j^-})\]
is spanned by the vectors $e_k^+$ ($1\leq k\leq i$) such that $\mathcal{G}(\omega)$ has a mark at $k^+$, the vectors $e_\ell^-$ ($1\leq \ell\leq j$) such that $\mathcal{G}(\omega)$ has a mark at $\ell^-$, and the linear combinations $e_k^++e_\ell^-$ ($1\leq k\leq i$ and $1\leq \ell\leq j$) such that $\mathcal{G}(\omega)$ has an edge joining the vertices $k^+$ and $\ell^-$. This implies that
\[
d_{i,j}(([\omega],\flag_0^+,\flag_0^-))=\dim [\omega]\cap(\Span{e_1^+,\ldots,e_i^+}+\Span{e_1^-,\ldots,e_j^-})=r_{i,j}(\omega).
\]
We deduce that
\begin{equation}
\label{3.1}
\Xorbit_\omega\subset \{\xi=(W,\flag^+,\flag^-)\in\Xfv: d_{i,j}(\xi)=r_{i,j}(\omega)\ \ \mbox{for all $i,j$}\}\quad\mbox{for all $\omega\in\parameters$}.
\end{equation}

If $\omega$, $\omega'$ are two different elements of the set $\parameters$, then their graphs $\mathcal{G}(\omega)$, $\mathcal{G}(\omega')$ must be different, hence the matrices $R(\omega)=(r_{i,j}(\omega))$ and $R(\omega')=(r_{i,j}(\omega'))$ are different. From (\ref{3.1}), it follows that the orbits $\Xorbit_\omega$ and $\Xorbit_{\omega'}$ are disjoint.
Therefore, $\Xfv$ is the disjoint union of the orbits $\Xorbit_\omega$ for $\omega\in\parameters$. This also implies that the inclusion in (\ref{3.1}) must be an equality for all $\omega\in\parameters$.
This establishes Theorem \ref{T1}\,(1) and (3).
\end{proof}

\begin{proof}[Proof of Theorem \ref{T1}\,(2)]
Lemma \ref{L3.1} implies
\begin{equation}
\label{dim-1}
\dim\Xorbit_\omega=\dim B_K\cdot[\omega]+\dim K/B_K=\dim B_K\cdot[\omega]+\binom{p}{2}+\binom{q}{2}.
\end{equation}
Let $\fb_K=\Lie(B_K)=\left\{\begin{pmatrix} x & 0 \\ 0 & y \end{pmatrix}:x\in\fb_p^+,\ y\in\fb_q^+\right\}$, where $\fb_p^+=\Lie(B_p^+)\subset\Mat_p(\C)$ and $\fb_q^+=\Lie(B_q^+)\subset\Mat_q(\C)$ are the subspaces of upper triangular matrices.
We have
\begin{eqnarray}
\dim B_K\cdot[\omega] & = & \dim B_K-\dim\{b\in B_K:b([\omega])=[\omega]\}  \label{dim-2} \\
 & = & \dim\fb_K-\dim\{z\in\fb_K:z([\omega])\subset[\omega]\}. \nonumber
\end{eqnarray}
As before, we denote by $(e_1^+,\ldots,e_p^+)$, resp. $(e_1^-,\ldots,e_q^-)$, the standard basis of $V^+=\C^p\times\{0\}^q$, resp. $V^-=\{0\}^p\times\C^q$.
The linear space $[\omega]$ is spanned by the vectors $e_a^+$ with $a\in\{1,\ldots,p\}$ such that the graph $\mathcal{G}(\omega)$ has a mark at $a^+$, the vectors $e_c^-$ with $c\in\{1,\ldots,q\}$ such that there is a mark at $c^-$, and the linear combinations $e_a^++e_c^-$ for all pairs $(a,c)\in\{1,\ldots,p\}\times\{1,\ldots,q\}$ such that $\mathcal{G}(\omega)$ has an edge joining $a^+$ and $c^-$.
This implies that a matrix $z=\begin{pmatrix} x & 0 \\ 0 & y \end{pmatrix}\in\fb_K$
satisfies $z([\omega])\subset [\omega]$ if and only if the upper triangular matrices $x=(x_{i,j})_{1\leq i,j\leq p}$ and $y=(y_{i,j})_{1\leq i,j\leq q}$ satisfy the following equations:
\begin{enumerate}
\item For every $a\in\{1,\ldots,p\}$ such that $\mathcal{G}(\omega)$ has a mark at $a^+$, we must have $x_{i,a}=0$ for all $i<a$ such that there is no mark at $i^+$.
\item For every $c\in\{1,\ldots,q\}$ such that $\mathcal{G}(\omega)$ has a mark at $c^-$, we must have $y_{j,c}=0$ for all $j<c$ such that there is no mark at $j^-$.
\item For every pair $(a,c)\in\{1,\ldots,p\}\times\{1,\ldots,q\}$ such that $\mathcal{G}(\omega)$ has an edge $(a^+,c^-)$, we must have $x_{i,a}=0$ for all $i<a$ such that $i^+$ is a free vertex (i.e. not marked nor incident with an edge) in $\mathcal{G}(\omega)$,
and we must have $y_{j,c}=0$ for all $j<c$ such that $j^-$ is a free vertex in $\mathcal{G}(\omega)$.
\item For $(a,c)$ as in (3), we must also have $x_{i,a}=y_{j,c}$ for all pair $(i,j)\in\{1,\ldots,a\}\times\{1,\ldots,c\}$ such that $(i^+,j^-)$ is an edge in $\mathcal{G}(\omega)$, i.e., for all edge which is situated on the left of $(a^+,c^-)$ or coincides with $(a^+,c^-)$ itself. Finally, we get one more equation $x_{i,a}=0$, or resp. $y_{j,c}=0$, for every edge $(i^+,j^-)$ which has a crossing with $(a^+,c^-)$, i.e., such that $i<a$ and $c<j$, resp. $i>a$ and $j<c$.
\end{enumerate}
We have listed linearly independent equations which characterize the subspace $\{z\in\fb_K:z([\omega])\subset[\omega]\}\subset\fb_K$. With the notation of Theorem \ref{T1}\,(2), the above items (1)--(3) yield $a^+(\omega)+a^-(\omega)$ equations, while the item (4) yields $\frac{b(\omega)(b(\omega)+1)}{2}+c(\omega)$ equations.
This implies that
\[
\dim\{z\in\fb_K:z([\omega])\subset[\omega]\}=\dim\fb_K-\Big(a^+(\omega)+a^-(\omega)+\frac{b(\omega)(b(\omega)+1)}{2}+c(\omega)\Big).
\]
Combining this equality with (\ref{dim-1}) and (\ref{dim-2}), we get the dimension formula stated in Theorem \ref{T1}\,(2).
\end{proof}

\subsection{Closure relations -- proof of Theorems \ref{T1}\,(4) and \ref{C1}}

\label{section-3.3}

For $\omega,\omega'\in\parameters$, we write $\omega\preceq\omega'$ if we have $r_{i,j}(\omega)\geq r_{i,j}(\omega')$ for all
$(i,j)\in\{0,\ldots,p\}\times\{0,\ldots,q\}$.
This clearly endows $\parameters$ with a partial order and, for showing Theorem \ref{T1}\,(4), we have to show that this order characterizes the inclusion relations between orbit closures in $\Xfv$.
We need four lemmas.

\begin{lemma}
\label{lemma-3.2}\label{L3.4}
Assume that $\omega$ is obtained from $\omega'$ by one of the elementary moves described in Figure \ref{figure1}. Then, the following relations hold:
\[\omega\prec \omega'\qquad\mbox{and}\qquad \Xorbit_\omega\subset\overline{\Xorbit_{\omega'}}.\]
\end{lemma}

\begin{proof}
We consider the cases described in Figure \ref{figure1}.
\begin{itemize}
\item
In Case 1, we have $r_{i,j}(\omega)=r_{i,j}(\omega')+1$ if $a\leq i<b$ and $c\leq j<d$, and we have $r_{i,j}(\omega)=r_{i,j}(\omega')$ otherwise.
Hence $r_{i,j}(\omega)\geq r_{i,j}(\omega')$ for all $i,j$, and this implies that $\omega\prec\omega'$.
\item In Case 2, upper subcase (resp., lower subcase), we have $r_{i,j}(\omega)=r_{i,j}(\omega')+1$ if $a\leq i<b$ and $j<c$ (resp., $i<a$ and $c\leq j<d$), and $r_{i,j}(\omega)=r_{i,j}(\omega')$ otherwise. Hence, again, we get $\omega\prec\omega'$.
\item In Case 3, upper subcase (resp., lower subcase), we have $r_{i,j}(\omega)=r_{i,j}(\omega')+1$ if $a\leq i<b$ and $c\leq j$ (resp., $a\leq i$ and $c\leq j<d$) and $r_{i,j}(\omega)=r_{i,j}(\omega')$ otherwise. Whence $\omega\prec\omega'$.
\item In Case 4, upper subcase (resp., lower subcase), we have $r_{i,j}(\omega)=r_{i,j}(\omega')+1$ if $a\leq i$ and $j<c$ (resp., $i<a$ and $c\leq j$) and $r_{i,j}(\omega)=r_{i,j}(\omega')$ otherwise. Once again this yields $\omega\prec\omega'$.
\item In Case 5, upper subcase (resp., lower subcase), we have $r_{i,j}(\omega)=r_{i,j}(\omega')+1$ if $a\leq i<b$ (resp., $c\leq j<d$) and $r_{i,j}(\omega)=r_{i,j}(\omega')$ otherwise, and once again we deduce that $\omega\prec\omega'$ in this case.
\end{itemize}
In each case, we have shown that $\omega\prec\omega'$.

As before, we denote by $e_{i,j}^k$ the $k\times k$ elementary matrix
with $1$ at position $(i,j)$ and $0$ elsewhere.
Then, let $u_{i,j}^k(t)=1_k+te_{i,j}^k$ and $\delta_i^k(t)=1_k+(t-1)e_{i,i}^k$.
For $t\in\C^*$, we consider the matrix $h_t$ given by
\[h_t=\begin{pmatrix} A_t & 0 \\ 0 & D_t \end{pmatrix},\]
where $A_t$ and $D_t$ are blocks of respective sizes $p\times p$ and $q\times q$ given by
\[
A_t=\left\{
\begin{array}{ll}
u_{a,b}^p(-t)\delta^p_a(t) & \mbox{in Cases 1, $2^+$,} \\
u_{a,b}^p(t) & \mbox{in Cases $3^+$, $5^+$,} \\
\delta_a^p(t) & \mbox{in Cases $3^-$, $4^+$,} \\
1_p & \mbox{in Cases $2^-$,\,$4^-$,\,$5^-$,} 
\end{array}
\right.
\
D_t=\left\{
\begin{array}{ll}
u^q_{c,d}(t)\delta^q_c(-t) & \mbox{in Cases 1, $2^-$,} \\
u_{c,d}^q(t) & \mbox{in Cases $3^-$, $5^-$,} \\
\delta_c^q(t) & \mbox{in Cases $3^+$, $4^-$,} \\
1_q & \mbox{in Cases $2^+$,\,$4^+$,\,$5^+$.} 
\end{array}
\right.
\]
Here the notation $N^+$ (resp., $N^-$) refers to the upper (resp., lower) subcase of Case $N$ in Figure \ref{figure1}.
In each case, we obtain a subset $\{h_t\}_{t\in\C^*}\subset K$ such that
\[([\omega],\flag_0^+,\flag_0^-)=\lim_{t\to \infty}h_t\cdot([\omega'],\flag_0^+,\flag_0^-),\]
and this shows that the inclusion $\Xorbit_\omega\subset\overline{\Xorbit_{\omega'}}$ holds.
\end{proof}

\begin{lemma}
\label{L3.5}
For every $\omega,\omega'\in\parameters$, the following implication holds:
\[\Xorbit_\omega\subset\overline{\Xorbit_{\omega'}}\quad\Longrightarrow\quad \omega\preceq\omega'.\]
\end{lemma}

\begin{proof}
Assume that $\Xorbit_\omega\subset\overline{\Xorbit_{\omega'}}$. For each pair $(i,j)\in\{0,\ldots,p\}\times\{0,\ldots,q\}$, the mapping
\[
\Xfv\to\mathbb{Z}_{\geq 0},\quad (W,(F^+_k)_{k=0}^p,(F^-_\ell)_{\ell=0}^q)\mapsto \dim W\cap(F_i^++F_j^-)
\]
is upper semicontinuous. Thus, in view of Theorem \ref{T1}\,(3), we have
\[
\Xorbit_\omega\subset \overline{\Xorbit_{\omega'}}\subset\{(W,(F^+_k)_{k=0}^p,(F^-_\ell)_{\ell=0}^q)\in\Xfv:\dim W\cap (F_i^++F_j^-)\geq r_{i,j}(\omega')\},
\]
whereas
\[
\Xorbit_\omega\subset\{(W,(F^+_k)_{k=0}^p,(F^-_\ell)_{\ell=0}^q)\in\Xfv:\dim W\cap (F_i^++F_j^-)=r_{i,j}(\omega)\}.
\]
This yields $r_{i,j}(\omega)\geq r_{i,j}(\omega')$ for all pair $(i,j)$, hence $\omega\preceq\omega'$.
\end{proof}

\begin{lemma}
\label{L3.6}
For every $\omega,\omega''\in\parameters$ such that $\omega\prec\omega''$, there is $\omega'\in\parameters$ with $\omega\preceq\omega'\preceq\omega''$ such that one of the pairs $(\omega,\omega')$, $(\omega',\omega'')$ fits in one of the cases described in Figure \ref{figure1}.
\end{lemma}

\begin{proof}
We reason by induction on $p+q\geq 0$, with immediate initialization if $p+q=0$.
Assume that $\omega,\omega''\in\parameters$ are such that $\omega\prec\omega''$. In particular we have $\omega\not=\omega''$,
which forces $r\geq 1$,
i.e., the graphs $\mathcal{G}(\omega)$ and $\mathcal{G}(\omega'')$ have at least one edge or marked vertex.

In the case where $\mathcal{G}(\omega)$ and $\mathcal{G}(\omega'')$ have 
one common edge or mark -- call it $x$,
by removing this edge or mark together with the corresponding vertices (or vertex), and after renumbering of the vertices, we obtain subgraphs $\mathcal{G}(\check\omega)=\mathcal{G}(\omega)\setminus x$ and $\mathcal{G}(\check\omega'')=\mathcal{G}(\omega'')\setminus x$ associated to smaller sized matrices $\check\omega$ and $\check\omega''$, and we still have $\check\omega\prec\check\omega''$ due to the definition of the relation $\preceq$. The induction hypothesis yields $\check\omega'$ with $\check\omega\preceq\check\omega'\preceq\check\omega''$, whose associated graph $\mathcal{G}(\check\omega')$ yields $\mathcal{G}(\check\omega)$ or is yielded by $\mathcal{G}(\check\omega'')$ through one of the elementary moves described in Figure \ref{figure1}. There is an element $\omega'\in\parameters$ such that $\mathcal{G}(\check\omega')=\mathcal{G}(\omega')\setminus x$, and this element satisfies the requirements of the lemma. In conclusion,
\begin{equation}
\label{no-common-edge}
\mbox{we may assume that $\mathcal{G}(\omega)$ and $\mathcal{G}(\omega'')$ have no common edge nor marked vertex.}
\end{equation}

\medskip
\noindent
{\it Notation:}
It is convenient to encode the set of edges and marks of the graph $\mathcal{G}(\omega)$ in the following way:
\begin{eqnarray*}
E(\omega) & := & \{(a,c)\in\{1,\ldots,p\}\times\{1,\ldots,q\}:\mbox{$(a^+,c^-)$ is an edge in $\mathcal{G}(\omega)$}\} \\
 & & \cup\{(a,0):a\in\{1,\ldots,p\},\ \mbox{$\mathcal{G}(\omega)$ has a mark at $a^+$}\} \\
 & & \cup\{(0,c):c\in\{1,\ldots,q\},\ \mbox{$\mathcal{G}(\omega)$ has a mark at $c^-$}\}.
\end{eqnarray*}
Then we note that
\begin{equation}
\label{E-and-rij}
r_{i,j}(\omega)=\card E(\omega)\cap(\{0,\ldots,i\}\times\{0,\ldots,j\})\quad\mbox{for all $i,j$}.
\end{equation}
We define the set $E(\omega'')$ relative to $\omega''$ in the same way.
Both sets $E(\omega)$ and $E(\omega'')$ have $r$ elements, in particular they are nonempty.

\smallskip

We choose an element $(a_0,c_0)\in E(\omega'') $ with the minimal possible value of $c_0$.
If $c_0\not=0$, then $c_0^-$ is not a free vertex in $\mathcal{G}(\omega'')$: it is marked if $a_0=0$ or incident with an edge $(a_0^+,c_0^-)$ if $a_0\not=0$. Moreover, the minimality of $c_0$ guarantees then that every vertex $c^-$ with $c<c_0$ is free in $\mathcal{G}(\omega'')$, and $\mathcal{G}(\omega'')$ contains no mark at $a^+$ for all $a\in\{1,\ldots,p\}$ (because $(a,0)$ cannot belong to $E(\omega'')$, due to the minimality of $c_0$).

If $c_0=0$, then $a_0^+$ is a marked vertex in $\mathcal{G}(\omega'')$.
There may be more than one marked vertex of this type in $\mathcal{G}(\omega'')$, and we choose $a_0$ minimal for this property. Thus, in any situation, we have $r_{a_0,c_0}(\omega'')=1$.

\smallskip
\noindent
{\it Case 1:} $c_0\not=0$ and $r_{a_0,c_0-1}(\omega)\geq 1$.

The condition means that $\mathcal{G}(\omega)$ has an edge or mark within the vertices $\{i^+:1\leq i\leq a_0\}\cup\{j^-:1\leq j<c_0\}$.
In other words, we can find a pair $(a_1,c_1)\in E(\omega)$
with $0\leq a_1\leq a_0$ and $0\leq c_1<c_0$.
Note that $(a_0,c_1)\not=(0,0)$ since $(a_1,c_1)\not=(0,0)$.
There is an element
$\omega'\in\parameters$ such that
\[
E(\omega')=(E(\omega'')\setminus\{(a_0,c_0)\})\cup\{(a_0,c_1)\}.
\]
This incorporates several situations, and in each one the graph $\mathcal{G}(\omega')$ is deduced from $\mathcal{G}(\omega'')$ through one of the elementary moves depicted in Figure \ref{figure1}:
\begin{itemize}
\item If $a_0\not=0$ and $c_1\not=0$ (resp., $c_1=0$), then $\mathcal{G}(\omega')$ is obtained from $\mathcal{G}(\omega'')$ by replacing
the edge $(a_0^+,c_0^-)$ by an edge $(a_0^+,c_1^-)$ (resp., by a mark
at $a_0^+$), whereas $c_0^-$ becomes a free vertex. This
corresponds to Case 3 -- lower subcase (resp., Case 4 -- upper
subcase) in Figure \ref{figure1}.
\item If $a_0=0$, then $\mathcal{G}(\omega'')$
has a mark at $c_0^-$, and $\mathcal{G}(\omega')$
is obtained by replacing this mark by a mark at $c_1^-$, whereas $c_0^-$ becomes a free vertex.
This corresponds to Case 5 -- lower subcase in Figure \ref{figure1}.
\end{itemize}
In each situation, we get $\omega'\prec\omega''$ in view of Lemma \ref{lemma-3.2}.
For $(i,j)\in\{0,\ldots,p\}\times\{0,\ldots,q\}$, we have
$r_{i,j}(\omega')=r_{i,j}(\omega'')$ (hence $r_{i,j}(\omega')\leq r_{i,j}(\omega)$) unless $i\geq a_0$ and $c_1\leq j<c_0$.
If $i\geq a_0$ and $c_1\leq j<c_0$, we have
\[r_{i,j}(\omega')=r_{i,j}(\omega'')+1=1\leq r_{i,j}(\omega)\]
(the second equality is due to the minimality of $c_0$, while the inequality follows from (\ref{E-and-rij}) and the fact that $(a_1,c_1)\in E(\omega)$). We conclude that the inequality $r_{i,j}(\omega')\leq r_{i,j}(\omega)$ holds for all pair $(i,j)$, hence $\omega\preceq\omega'$, and the element $\omega'$ satisfies all the requirements of the lemma.

\smallskip
\noindent
{\it Case 2:} $c_0=0$ or $r_{a_0,c_0-1}(\omega)=0$.

This condition implies that the set $E(\omega)$ contains no pair of the form $(a,c)$ with $0\leq a\leq a_0$ and $0\leq c<c_0$ (see (\ref{E-and-rij})).
Note also that $(a_0,c_0)\notin E(\omega)$ (due to (\ref{no-common-edge})).

Since $r_{a_0,c_0}(\omega)\geq r_{a_0,c_0}(\omega'')=1$, there is a pair $(a'_0,c_0)\in E(\omega)$ with $0\leq a'_0< a_0$.
In particular this forces $a_0\not=0$.

The fact that $a_0\not=0$ implies that $a_0^+$ is a vertex in $\mathcal{G}(\omega)$. Either $a_0^+$ is incident with an edge/marked in $\mathcal{G}(\omega)$, in which case $E(\omega)$ contains an element of the form $(a_0,d_0)$ with $c_0< d_0\leq q$, or $a_0^+$ is a free vertex in $\mathcal{G}(\omega)$, in which case we set $d_0=q+1$.

We choose $a_1\in\{a'_0,\ldots,a_0-1\}$ maximal such that $(a_1,c_1)\in E(\omega)$ for some $c_1$ with $c_0\leq c_1<d_0$.
There is an element
$\omega'\in\parameters$ such that
\[
E(\omega')=\left\{\begin{array}{ll}
(E(\omega)\setminus\{(a_1,c_1),(a_0,d_0)\})\cup\{(a_1,d_0),(a_0,c_1)\} & \mbox{if $d_0\leq q$}, \\
(E(\omega)\setminus\{(a_1,c_1)\})\cup\{(a_0,c_1)\} & \mbox{if $d_0=q+1$}.
\end{array}
\right.
\]
In each situation, the graph $\mathcal{G}(\omega')$ yields $\mathcal{G}(\omega)$ by one of the moves of Figure \ref{figure1}:
\begin{itemize}
\item In the case where $d_0\leq q$,
the graph $\mathcal{G}(\omega)$ has an edge $(a_0^+,d_0^-)$. If
$a_1,c_1\not=0$, then $(a_1^+,c_1^-)$ is also an edge in
$\mathcal{G}(\omega)$, and the relation between
$\mathcal{G}(\omega')$ and $\mathcal{G}(\omega)$ is as depicted in
Case 1 of Figure \ref{figure1}. If $c_1=0$ (resp., $a_1=0$), then
$\mathcal{G}(\omega)$ has a mark at $a_1^+$ (resp., $c_1^-$) and the
relation with $\mathcal{G}(\omega')$ is as in Case 2 - upper subcase
(resp., lower subcase) of Figure \ref{figure1}.
\item In the case where $d_0=q+1$, the vertex $a_0^+$ is a free vertex in $\mathcal{G}(\omega)$. The relation between $\mathcal{G}(\omega')$ and $\mathcal{G}(\omega)$ is as described in Case 3 - upper subcase, Case 5 - upper subcase, or Case 4 - lower subcase of Figure \ref{figure1}, depending on whether
$a_1,c_1\not=0$, $c_1=0$ (and $a_1\not=0$), or $a_1=0$ (and $c_1\not=0$).
\end{itemize}
In particular we have $\omega\prec\omega'$ (by Lemma \ref{lemma-3.2}).

For all $(i,j)\in\{0,\ldots,p\}\times\{0,\ldots,q\}$, we have $r_{i,j}(\omega')=r_{i,j}(\omega)$ unless $a_1\leq i<a_0$ and $c_1\leq j<d_0$, in which case we have $r_{i,j}(\omega')=r_{i,j}(\omega)-1$. In the latter situation, we nevertheless have
\[
r_{i,j}(\omega)=r_{a_0,j}(\omega)\geq r_{a_0,j}(\omega'')=1+r_{a_0-1,j}(\omega'')\geq 1+r_{i,j}(\omega'')
\]
(where the first equality is due to the maximality of $a_1$). Thus, the inequality $r_{i,j}(\omega')\geq r_{i,j}(\omega'')$ holds for all pair $(i,j)$, and therefore we have $\omega'\preceq\omega''$. The element $\omega'$ satisfies the required conditions. This completes the proof of the lemma.
\end{proof}

\begin{lemma}
\label{L3.7}
If $\Xorbit_{\omega'}$ covers $\Xorbit_\omega$, then $\dim\Xorbit_{\omega'}=\dim\Xorbit_\omega+1$.
\end{lemma}

\begin{proof}
Note that $\Xorbit_{\omega'}$ covers $\Xorbit_\omega$ if and only if $\overline{\Xorbit_\omega}$ is an irreducible component of $\overline{\Xorbit_{\omega'}}\setminus\Xorbit_{\omega'}$.
The conclusion of the lemma is implied by the following general fact, taking also Lemma \ref{L3.1} into account.

\smallskip
\noindent
{\it Fact:} Given a connected solvable algebraic group acting on an algebraic variety, the boundary $\partial O=\overline{O}\setminus O$ of each (non closed) orbit is equidimensional of codimension $1$ in $\overline{O}$.

\smallskip
\noindent
A proof of this fact can be found in \cite[Lemmas 2.12--2.13]{Timashev}.
\end{proof}

Now we are in position to proceed with the proof of Theorems \ref{T1}\,(4) and \ref{C1}.

\begin{proof}[Proof of Theorem \ref{T1}\,(4)] The ``only if'' part is shown in Lemma \ref{L3.5}. For the inverse implication, let $\omega,\omega'\in\parameters$ be such that $\omega\preceq\omega'$. Repeated applications of Lemma \ref{L3.6}
yield a sequence of elements
\[\omega=\omega_0\prec \omega_1\prec\cdots\prec \omega_\ell=\omega'\]
such that $(\omega_{k-1},\omega_k)$ fits in one of the cases of Figure \ref{figure1} for all $k$.
Then, Lemma \ref{L3.4} shows that the following sequence of inclusions holds:
\[\overline{\Xorbit_{\omega_0}}\subset \overline{\Xorbit_{\omega_1}}\subset\cdots\subset \overline{\Xorbit_{\omega_\ell}}.\]
In particular, we get the desired inclusion $\overline{\Xorbit_\omega}\subset\overline{\Xorbit_{\omega'}}$.
\end{proof}

\begin{proof}[Proof of Theorem \ref{C1}]
Assume that condition (1) of Theorem \ref{C1} holds.
First, the equality
$\dim\Xorbit_{\omega'}=\dim\Xorbit_\omega+1$
follows from Lemma \ref{L3.7}.
Next, we have in particular $\omega\prec\omega'$ in view of Theorem \ref{T1}\,(4). Lemma \ref{L3.6} yields an element $\omega_0\in\parameters$ with $\omega\preceq\omega_0\preceq\omega'$ and such that $(\omega,\omega_0)$ or $(\omega_0,\omega')$ fits in one of the cases of Figure \ref{figure1}. By Theorem \ref{T1}\,(4) again, we get $\overline{\Xorbit_\omega}\subset\overline{\Xorbit_{\omega_0}}\subset\overline{\Xorbit_{\omega'}}$, and therefore $\omega=\omega_0$ or $\omega_0=\omega'$, due to the assumption that $\Xorbit_{\omega'}$ covers $\Xorbit_\omega$. In both cases this implies that the pair $(\omega,\omega')$ fits in one of the cases of Figure \ref{figure1}, that is, the graph $\mathcal{G}(\omega)$ is obtained from $\mathcal{G}(\omega')$ by one of the moves listed in Figure \ref{figure1}. This yields condition (2) of Theorem \ref{C1}.

Conversely, assume (2). By Lemma \ref{L3.4}, the inclusion $\overline{\Xorbit_\omega}\subset\overline{\Xorbit_{\omega'}}$ holds. This inclusion, combined with the fact that $\dim\Xorbit_{\omega'}=\dim\Xorbit_\omega+1$, implies that $\Xorbit_{\omega'}$ covers $\Xorbit_\omega$.
\end{proof}

\subsection{Further remarks}

\label{section-3.4}

(a)
In view of Lemma \ref{L3.1}, the results shown in Section \ref{section-3} establish the properties of
the $K$-orbits on $\Xfv$ as well as of the $B_K$-orbits on $\Grass(V,r)$.
Specifically, we obtain the decomposition
\[
\Grass(V,r)=\bigsqcup_{\omega\in\parameters} B_K\cdot[\omega].
\]
Note that $\Grass(V,r)$ is a fortiori a union of finitely many orbits for the action of $K$.
The description of these $K$-orbits is well known,
and it can be related to the decomposition into $B_K$-orbits in the following way.

Given $\omega\in\parameters$, we have introduced a matrix $R(\omega)=(r_{i,j}(\omega))_{\substack{0\leq i\leq p \\ 0\leq j\leq q}}$ which determines the orbit $B_K\cdot[\omega]$ (Theorem \ref{T1}\,(3)) and its closure relations with other orbits (Theorem \ref{T1}\,(4)).
In particular, the pair of integers $(r_{p,0}(\omega),r_{0,q}(\omega))$ can be expressed as
\[
(r_{p,0}(\omega),r_{0,q}(\omega))=(\dim [\omega]\cap V^+,\dim [\omega]\cap V^-),
\]
and we have actually
\[
K\cdot[\omega]=\{W\in\Grass(V,r):(\dim W\cap V^+,\dim W\cap V^-)=(r_{p,0}(\omega),r_{0,q}(\omega))\}.
\]
Thus
\[
K\cdot[\omega]=K\cdot[\omega']\Longleftrightarrow (r_{p,0}(\omega),r_{0,q}(\omega))=(r_{p,0}(\omega'),r_{0,q}(\omega')).
\]
Moreover,
\[
K\cdot[\omega]\subset\overline{K\cdot[\omega']}\Longleftrightarrow \big(r_{p,0}(\omega)\geq r_{p,0}(\omega')\quad\mbox{and}\quad r_{0,q}(\omega)\geq r_{0,q}(\omega')\big).
\]

Note that the number $s:=r_{p,0}(\omega)$ (resp.,
$t:=r_{0,q}(\omega)$) is the number of marks among the positive
(resp., negative) vertices of the graph $\mathcal{G}(\omega)$. 
In view of Theorem \ref{T1}\,(2), the
$B_K$-orbit $B_K\cdot[\omega]$ is dense in its $K$-saturation
$K\cdot[\omega]$ if and only if the degree of vertices is nondecreasing
from left to right along the row of positive (resp., negative)
vertices of $\mathcal{G}(\omega)$ (i.e., marked vertices are located on the
right and free vertices are located
on the left), and each pair of edges has a crossing. Thus there are
$\binom{k}{2}$ crossings, where $k:=r-(s+t)$ is the number of edges,
and we have
\[
\dim K\cdot[\omega]=\dim B_K\cdot[\omega]=(s+k)(p-s)+(t+k)(q-t)-k^2. 
\]

For example, if $p=q=r=2$, the variety $\Grass(V,r)$ is the union of six $K$-orbits.
In Figure \ref{figure5} we indicate the graphs $\mathcal{G}(\omega)$ corresponding to
the $B_K$-orbits $B_K\cdot[\omega]$ which are dense in their $K$-saturation $K\cdot[\omega]$,
the dimensions of the $K$-orbits, and the cover relations; this diagram is deduced from Figure \ref{figure2} given in Example \ref{E2.4}.
\begin{figure}[h]
\textcolor{blue}{\xymatrixcolsep{1pc}
\xymatrix{
\textcolor{darkgray}{\dim:4}
& & & & & \graphA \ar@/_/@{-}[lld] \ar@/^/@{-}[rrd]  & & & & & \\
\textcolor{darkgray}{3} & & & \graphB \ar@{-}[dd] \ar@{-}[rrd] & & & & \graphD \ar@{-}[lld] \ar@{-}[dd] & & & \\
\textcolor{darkgray}{2} & & & & & \graphG & & & & & \\
\textcolor{darkgray}{0} & & & \graphN & & & & \graphP & & &}}
\caption{The parameters of the $K$-orbits of $\Grass(V,r)$ and the cover relations for $p=q=r=2$.}
\label{figure5}
\end{figure}

(b) In \cite{Matsuki-Oshima},
Matsuki and Oshima classify the orbit set $K\backslash G/B$, where $B\subset G=\GL_{p+q}(\C)$ is a Borel subgroup.
It appears that our parametrization of $\Xfv/K\cong B_K\backslash G/P$ ressembles to theirs;
here $P\subset G$ denotes the maximal parabolic subgroup obtained as the stabilizer of an $r$-dimensional subspace, and which satisfies $B\subset P$.
In particular, by gluing orbits, Matsuki and Oshima's classification yields a parametrization of $K\backslash G/P$
which coincides with the one described in part (a) above.
We can speculate on a deeper relation between the two orbit sets
$K\backslash G/B$ and $B_K\backslash G/P$.
We also refer to \cite{Trapa-IMRN}, where the image of the moment map for the action of $K$ on $G/B$ is considered.


\section{Calculation of symmetrized and exotic Steinberg maps}

\label{section-4}

\subsection{Conormal direction}

As shown in Theorem \ref{T1}, every $K$-orbit in $\Xfv$ takes the form
\[
\Xorbit_\omega=K\cdot([\omega],\flag_0^+,\flag_0^-)
\]
for a matrix $\omega=\begin{pmatrix} \tau_1 \\ \tau_2 \end{pmatrix}\in \parameters=\ppermutationsof_{(p,q),r}/\permutationsof{r}$,
where $[\omega]\in\Grass(V,r)$ stands for the image of $\omega$ and $(\flag_0^+,\flag_0^-)\in\Flags(V^+)\times \Flags(V^-)$ is the pair of standard flags.
With the notation of Section \ref{section-1.2} we have
\[\nil([\omega])=\{x\in\gl(V):\mathrm{Im}\,x\subset[\omega]\subset\mathrm{Ker}\,x\},\quad \nil(\flag_0^+)=\fn_p^+,\quad \mbox{and}\quad \nil(\flag_0^-)=\fn_q^+\]
where $\fn_k^+\subset\gl_k(\C)$ denotes the subalgebra of strictly upper-triangular matrices.
Hence, the conormal bundle to the orbit $\Xorbit_\omega$ is obtained as
\[T^*_{\Xorbit_\omega}\Xfv=K\cdot\{([\omega],\flag_0^+,\flag_0^-,x):x\in\condir_\omega\}\]
where
\begin{equation}
\label{Domega}
\condir_\omega:=\left\{x=\begin{pmatrix} a & b \\ c & d\end{pmatrix}\in\gl(V):(a,d)\in\fn_p^+\times\fn_q^+,\ \mathrm{Im}\,x\subset[\omega]\subset\mathrm{Ker}\,x\right\}.
\end{equation}
This immediately implies:

\begin{lemma}
\label{L4.1}
Let $\omega\in\parameters$.
Then, $\Phi_\fk(\Xorbit_\omega)$, respectively $\Phi_\fs(\Xorbit_\omega)$,
is characterized as being the unique $K$-orbit of $\nilpotentsof{\fk}$, resp. $\nilpotentsof{\fs}$, which intersects the space $\{x_\fk:x\in\condir_\omega\}$, resp. $\{x_\fs:x\in\condir_\omega\}$, along a dense open subset.
\end{lemma}

For the computation of the maps $\Phi_\fk$ and $\Phi_\fs$, we need a more detailed description of the conormal direction $\condir_\omega$.
We use the following notation: if $a$ is a $k\times \ell$ matrix, then for every subsets $R\subset\{1,\ldots,k\}$ and $S\subset\{1,\ldots,\ell\}$, we denote by $\subm{a}{R,S}$ the submatrix of $a$ formed by the coefficients $a_{i,j}$ with $i\in R$, $j\in S$, and we view it as a linear map from $\C^S$ to $\C^R$. Recall from Section \ref{section-2.1}
that $\omega$ gives rise to decompositions $\intp:=\{1,\ldots,p\}=I\cup L\cup L'$ and $\intq:=\{1,\ldots,q\}=J\cup M\cup M'$ and to a bijection $\sigma:J\to I$ which we view as a linear map from $\C^J$ to $\C^I$.

\begin{lemma}
\label{L4.2}
A matrix $x=\begin{pmatrix} a & b \\ c & d \end{pmatrix}$ (with $a\in\fn_p^+$, $d\in\fn_q^+$) belongs to $\mathcal{D}_\omega$ if and only if it satisfies the following conditions:
\begin{eqnarray}
&& \left\{\begin{array}{c} \subm{a}{\intp,L}=0,\ \subm{c}{\intq,L}=0,\ \subm{b}{\intp,M}=0,\ \subm{d}{\intq,M}=0,\\[1.5mm]
{} \subm{b}{\intp,J}=-\subm{a}{\intp,I}\sigma,\ \subm{d}{\intq,J}=-\subm{c}{\intq,I}\sigma, \end{array}\right. \label{4.1} \\[2mm]
&& \left\{\begin{array}{c} \subm{a}{L',\intp}=0,\ \subm{b}{L',\intq}=0,\ \subm{c}{M',\intp}=0,\ \subm{d}{M',\intq}=0, \\[1.5mm]
{} \subm{a}{I,\intp}=\sigma \subm{c}{J,\intp},\ \subm{b}{I,\intq}=\sigma\subm{d}{J,\intq}.
\end{array}\right. \label{4.2}
\end{eqnarray}
\end{lemma}

\begin{proof}
Let $(e_1^+,\ldots,e_p^+)$ be the standard basis of $V^+=\C^p\times\{0\}^q$ and let $(e_1^-,\ldots,e_q^-)$ be the standard basis of $V^-=\{0\}^p\times\C^q$. Then
\begin{equation}
\label{4.3}
[\omega]=\Span{e_i^+:i\in L;\quad e_j^-:j\in M;\quad e_{\sigma(j)}^++e_j^-:j\in J}.
\end{equation}
For every matrix $x$ such that $a\in\fn_p^+$ and $d\in\fn_q^+$, we have $x\in\condir_\omega$ if and only if $\mathrm{Im}\,x\subset[\omega]\subset \mathrm{Ker}\,x$, and in view of (\ref{4.3}) the second inclusion is equivalent to (\ref{4.1}) while the first inclusion is equivalent to (\ref{4.2}).
\end{proof}

Figure \ref{figure4} illustrates the form of the elements in the conormal direction $\mathcal{D}_\omega$. The matrix is represented blockwise with blocks indicating the submatrices $\subm{X}{R,S}$ relative to the subsets
$R,S\in\{I,J,L,L',M,M'\}$. Note that in the decompositions $I\cup L\cup L'=\intp$ and $J\cup M\cup M'=\intq$ the subsets are not consecutive,
and the matrix is thus represented modulo permutation within the rows and the columns. In particular, it is required in addition that the blocks $a$ and $d$ be strictly upper triangular.
\begin{figure}[h]
\[
\begin{blockarray}{ccccccc}
{} & {\quad I\quad} & {\quad L\quad} & {\quad L'\quad} & {\quad J\quad} & {\quad M\quad} & {\quad M'\quad} \\
\begin{block}{c(c|c|c||c|c|c)}
I & \sigma c_1 & 0 & \sigma c_2 & -\sigma c_1\sigma & 0 & \sigma d_2 \\
\cline{2-7}
L & a_3 & 0 & a_4 & -a_3\sigma & 0 & b_4 \\
\cline{2-7}
L' & 0 & 0 & 0 & 0 & 0 & 0 \\
\cline{2-7} \\[-12pt] \cline{2-7}
J & c_1 & 0 & c_2 & -c_1\sigma & 0 & d_2 \\
\cline{2-7}
M & c_3 & 0 & c_4 & -c_3\sigma & 0 & d_4 \\
\cline{2-7}
M' & 0 & 0 & 0 & 0 & 0 & 0 \\
\end{block}
\end{blockarray}
\]
\caption{Form of the elements $x=\left({}^{a \ b}_{c \ d}\right)$ (with $a\in\fn_p^+$, $d\in\fn_q^+$) belonging to the conormal direction $\condir_\omega$.}
\label{figure4}
\end{figure}

\subsection{A review of the orbit sets, and an involution}
As explained in Section \ref{section-2.3}, the orbits of $K$ in the nilpotent cone
\[
\nilpotentsof{\fk}\subset\fk=\left\{\begin{pmatrix} a & 0 \\ 0 & d \end{pmatrix}: (a,d)\in \gl_p(\C)\times\gl_q(\C)\right\}
\]
are parametrized by pairs of partitions $(\lambda,\mu)\in\partitionsof{p}\times\partitionsof{q}$, and we denote by $\Nkorbit_{\lambda,\mu}$ the orbit corresponding to the pair $(\lambda,\mu)$.
Note that
\begin{equation}
\label{4.5new}
\begin{pmatrix} a & 0 \\ 0 & d\end{pmatrix}\in\Nkorbit_{\lambda,\mu}\Longleftrightarrow \begin{pmatrix} d & 0 \\ 0 & a \end{pmatrix}\in\Nkorbit_{\mu,\lambda},
\end{equation}
though here the notation $\Nkorbit_{\mu,\lambda}$ refers to an orbit of $K^*:=\GL_q(\C)\times\GL_p(\C)$ on $\gl_q(\C)\times\gl_p(\C)$.

Recall also from Section \ref{section-2.3} that the orbits of $K$ in the nilpotent cone
\[
\nilpotentsof{\fs}\subset\fs=\left\{\begin{pmatrix} 0 & b \\ c & 0 \end{pmatrix}: (b,c)\in \Mat_{p,q}(\C)\times\Mat_{q,p}(\C)\right\}
\]
are parametrized by signed Young diagrams of signature $(p,q)$, and we denote by $\Nsorbit_\Lambda$ the orbit corresponding to $\Lambda$. We have
\begin{equation}
\label{4.6}
\begin{pmatrix} 0 & b \\ c & 0 \end{pmatrix}\in\Nsorbit_\Lambda\Longleftrightarrow \begin{pmatrix} 0 & c \\ b & 0 \end{pmatrix}\in\Nsorbit_{\Lambda^*}
\end{equation}
where $\Lambda^*$ denotes the signed Young diagram of signature $(q,p)$ obtained from $\Lambda$ by switching the $+$'s and the $-$'s, and $\Nsorbit_{\Lambda^*}$ is an orbit of the group $K^*$.

As shown in Theorem \ref{T1}, the orbits of $K$ in
\[
\Xfv=\Grass(V,r)\times\Flags(V^+)\times\Flags(V^-)
\]
are parametrized by the elements of $\parameters$, and
$\Xorbit_\omega$ is the orbit corresponding to $\omega$. If
$\omega=\begin{pmatrix} \tau_1 \\ \tau_2 \end{pmatrix}$ is an
element of
$\parameters=\ppermutationsof_{(p,q),r}/\permutationsof{r}$, then
$\omega^*:=\begin{pmatrix} \tau_2 \\ \tau_1 \end{pmatrix}$ is an
element of
$\parameters^*:=\ppermutationsof_{(q,p),r}/\permutationsof{r}$ which
thus yields an orbit $\Xorbit_{\omega^*}$ of $K^*$ in a suitable
multiple flag variety $\Xfv^*$. The graphic representation
$\mathcal{G}(\omega^*)$ of $\omega^*$ is obtained from
$\mathcal{G}(\omega)$ by switching the two rows of vertices, i.e.,
by relabeling every vertex $i^+$ (resp., $j^-$) as $i^-$ (resp.,
$j^+$). This implies that, if $(I,J,L,L',M,M',\sigma)$ are the data
corresponding to $\omega$ in the sense of Section \ref{section-2.1},
then the relevant data for $\omega^*$ are
\begin{equation}
\label{4.7}
(I^*,J^*,L^*,L'^*,M^*,M'^*,\sigma)=(J,I,M,M',L,L',\sigma^{-1}).
\end{equation}

Finally, by the description of the conormal direction in (\ref{Domega}), we have
\begin{equation}
\label{4.8}
\begin{pmatrix} a & b \\ c & d \end{pmatrix}\in\condir_\omega\Longleftrightarrow
\begin{pmatrix} d & c \\ b & a \end{pmatrix}\in\condir_{\omega^*}.
\end{equation}

The observations made in (\ref{4.5new}), (\ref{4.6}), (\ref{4.8}) combined with Lemma \ref{L4.1} yield the following statement:

\begin{lemma}
\label{L4.3}
\begin{enumerate}
\item If $\Phi_\fk(\Xorbit_\omega)=\Nkorbit_{\lambda,\mu}$, then $\Phi_\fk(\Xorbit_{\omega^*})=\Nkorbit_{\mu,\lambda}$.
\item If $\Phi_\fs(\Xorbit_\omega)=\Nsorbit_\Lambda$, then $\Phi_\fs(\Xorbit_{\omega^*})=\Nsorbit_{\Lambda^*}$.
\end{enumerate}
\end{lemma}

There is an abuse of notation in that statement, since we use the notation $\Phi_\fk$ and $\Phi_\fs$ to designate also the symmetrized and exotic Steinberg maps relative to the symmetric pair $(G,K^*)=(\GL_{p+q}(\C),\GL_q(\C)\times\GL_p(\C))$.

\subsection{Symmetrized Steinberg map $\Phi_\fk$}

For a permutation $w\in\permutationsof{k}$ we consider the space
\[
\fn_k^+\cap({}^w\fn_k^+):=\fn_k^+\cap(\mathrm{Ad}(w)(\fn_k^+))=\{a\in\fn_k^+:w^{-1}aw\in\fn_k^+\}.
\]
The following is a well-known fact from classical Steinberg theory.

\begin{theorem}[\cite{Steinberg-occurrence}]
\label{T-Steinberg}
The unique nilpotent orbit $\mathcal{O}_\lambda\subset \gl_k(\C)$ which intersects the space $\fn_k^+\cap({}^w\fn_k^+)$ along a dense open subset is the one corresponding to the Young diagram
$\lambda=\shape(\RS_1(w))=\shape(\RS_2(w))$, where $(\RS_1(w),\RS_2(w))$ denotes the pair of Young tableaux associated to $w$ via the Robinson--Schensted correspondence.
\end{theorem}

In our situation, we consider the permutations $w_{\fk,+}\in\permutationsof{p}$ and $w_{\fk,-}\in\permutationsof{q}$
of (\ref{wk+}) and (\ref{wk-}), associated to an element $\omega\in\parameters$.

\begin{lemma}
\label{L4.5}
Let $a\in\fn_p^+$. The following conditions are equivalent:
\begin{enumerate}
\item There are matrices $b,c,d$ such that $x:=\begin{pmatrix} a & b \\ c & d\end{pmatrix}\in\condir_\omega$;
\item $\sigma^{-1}\subm{a}{I,I}\sigma$ is strictly upper triangular, $\subm{a}{L',\intp}=0$, and $\subm{a}{\intp,L}=0$;
\item $a\in \fn_p^+\cap({}^{w_{\fk,+}}\fn_p^+)$.
\end{enumerate}
\end{lemma}

\begin{proof}
The equivalence between (1) and (2) is implied by Lemma \ref{L4.2} (see Figure \ref{figure4}).
Given $a\in\fn_p^+$, condition (2) is equivalent to:
\begin{equation*}
\left(1\leq i<j\leq p\quad\mbox{and}\quad\left\{\begin{array}{l}
i\in L' \\ \mbox{or}\quad j\in L\\
\mbox{or}\quad(i,j\in I\ \ \mbox{and}\ \ \sigma^{-1}(i)>\sigma^{-1}(j))
\end{array}\right.\right)
\quad\implies\quad a_{i,j}=0.
\end{equation*}
Condition (3) is equivalent to:
\begin{equation*}
1\leq i<j\leq p\quad \mbox{and}\quad w_{\fk,+}^{-1}(i)> w_{\fk,+}^{-1}(j)\quad \implies\quad  a_{i,j}=0.
\end{equation*}
By definition of $w_{\fk,+}$ (see (\ref{wk+})), for $1\leq i<j\leq p$, we have
\[w_{\fk,+}^{-1}(i)> w_{\fk,+}^{-1}(j)\quad\iff\quad \left\{\begin{array}{l}
i\in L' \\ \mbox{or}\quad j\in L\\
\mbox{or}\quad(i,j\in I\ \ \mbox{and}\ \ \sigma^{-1}(i)>\sigma^{-1}(j)).
\end{array}\right.\]
Therefore, (2) and (3) are equivalent.
\end{proof}

\begin{proof}[Proof of Theorem \ref{T2}\,(1)]
Let
\[\Phi_\fk(\Xorbit_\omega)=\Nkorbit_{\lambda,\mu}=\left\{\begin{pmatrix} a & 0 \\ 0 & d \end{pmatrix}:(a,d)\in\mathcal{O}_\lambda\times\mathcal{O}_\mu\subset\gl_p(\C)\times\gl_q(\C)\right\}.\]
By Lemmas \ref{L4.1} and \ref{L4.5}, the nilpotent orbit $\mathcal{O}_\lambda\subset\gl_p(\C)$ is characterized as being the unique $\GL_p(\C)$-orbit which intersects the space
\[\left\{a\in\fn_p^+:\exists b,c,d\mbox{ such that }\begin{pmatrix} a & b \\ c & d \end{pmatrix}\in\condir_\omega\right\}=\fn_p^+\cap({}^{w_{\fk,+}}\fn_p^+)\]
along a dense open subset. By Theorem \ref{T-Steinberg}, this implies that
\[\lambda=\shape(\RS_1(w_{\fk,+})).\]
By (\ref{4.7}) and Lemma \ref{L4.3}, we also deduce that
\[\mu=\shape(\RS_1(w_{\fk,-})).\]
The proof of Theorem \ref{T2}\,(1) is complete.
\end{proof}

\subsection{Exotic Steinberg map $\Phi_\fs$}

We introduce notation which extend our notation on matrices.
Given subsets $R,S$ of integers, we let $\Mat_{R,S}(\C)$ denote the space of linear homomorphisms $x:\C^S\to \C^R$, which can be viewed as well as matrices of coefficients $(x_{i,j})_{(i,j)\in R\times S}$.
Let $\fn_R^+\subset\Mat_{R,R}(\C)$ be the subspace of endomorphisms that are strictly upper triangular as matrices. 

If $R,S$ are respectively subsets of $R',S'$, we denote by
\begin{equation}
\label{eta}
\eta_{R,S}^{R',S'}:\Mat_{R,S}(\C)\to\Mat_{R',S'}(\C),\quad x\mapsto \hat{x}=(\hat{x}_{i,j})_{(i,j)\in R'\times S'}
\end{equation}
the linear morphism which maps a matrix $x$ to its {\em extension by zero} given by
$\hat{x}_{i,j}=x_{i,j}$ if $(i,j)\in R\times S$ and $\hat{x}_{i,j}=0$ otherwise.

A bijection $w:S\to R$ yields an element of $\Mat_{R,S}(\C)$ also denoted by $w$ by abuse of notation. In addition, through the Robinson--Schensted algorithm, $w$ gives rise to a pair of Young tableaux $(\RS_1(w),\RS_2(w))$ of same shape, whose respective sets of entries are $R$ and $S$.

The following is a reformulation of Theorem \ref{T-Steinberg}.

\begin{proposition}
\label{P4.6}
Given a bijection $w: S\to R$, the Jordan normal form of a general element $x$ in the space
\[
\fn_R^+\cap({}^w\fn_S^+):=\{x\in\fn_R^+:w^{-1}xw\in\fn_S^+\}
\]
is given by the Young diagram $\shape(\RS_1(w))=\shape(\RS_2(w))$. In other words, for all $k\geq 1$, $\dim \ker x^k$ is the number of boxes in the first $k$ columns of $\RS_1(w)$.\end{proposition}

Take an element $\omega\in\parameters$ with corresponding data $(I,J,L,L',M,M',\sigma)$. As in Section \ref{section-2.3}, we write $J=\{j_1<\ldots<j_k\}$, $L'=\{\ell'_1<\ldots<\ell'_{s'}\}$, $M=\{m_1<\ldots<m_t\}$.
Moreover, we denote
\[
S=J\cup M\cup \{q+1,\ldots,q+s'\}\quad\mbox{and}\quad
 R=\{-t,\ldots,-1\}\cup I\cup L'.
\]
We will consider the bijection $w:=w_{\fs,+}:S\to R$ defined in (\ref{ws+}).

We denote $\varsigma=\eta_{I,J}^{I\cup L',J\cup M}(\sigma)$ and $\tau=\eta_{I,J}^{\intp,\intq}(\sigma)=\eta_{I\cup L',J\cup M}^{\intp,\intq}(\varsigma)$ (see (\ref{eta})), that is,
\begin{equation}
\label{zeta-tau}
\varsigma=\begin{blockarray}{ccc}
{} & {\ J\ } & {\ M\ } \\
\begin{block}{c(c|c)}
I & \sigma & 0 \\
\cline{2-3}
L' & 0 & 0  \\
\end{block}
\end{blockarray}\,,\qquad \tau=\begin{blockarray}{cccc}
{} & {\ J\ } & {\ M\ } & {\ M'\ } \\
\begin{block}{c(c|c|c)}
I & \sigma & 0 & 0 \\
\cline{2-4}
L & 0 & 0 & 0 \\
\cline{2-4}
L' & 0 & 0 & 0 \\
\end{block}
\end{blockarray}
\end{equation}
(after changing the order of rows and columns).

\begin{lemma}
\label{L4.8}
Let $x=\begin{pmatrix} a & b \\ c & d \end{pmatrix}\in\condir_\omega$, so that $x_\fk=\begin{pmatrix} a & 0 \\ 0 & d\end{pmatrix}$ and $x_\fs=\begin{pmatrix} 0 & b \\ c & 0\end{pmatrix}$. Then, for every $m\geq 0$, we have
\[
(x_\fs)^{2m}=(-1)^m(x_\fk)^{2m}=(-1)^m\begin{pmatrix} a^{2m} & 0 \\ 0 & d^{2m} \end{pmatrix}
\quad\mbox{and}\quad
(x_\fs)^{2m+1}=(-1)^m\begin{pmatrix} 0 & * \\ (c\tau)^{2m}c & 0 \end{pmatrix}
\]
(where the symbol $*$ stands for some matrix in $\Mat_{p,q}(\C)$, whose precise description is not needed).
\end{lemma}

\begin{proof}
Every element $x\in\condir_\omega$ is such that $x^2=0$. On the other hand, we can see that $(x^2)_\fk=(x_\fk)^2+(x_\fs)^2$. The formula for $(x_\fs)^{2m}$ ensues.

It readily follows that
\[(x_\fs)^{2m+1}=(-1)^m\begin{pmatrix} 0 & a^{2m}b \\ d^{2m}c & 0 \end{pmatrix}.\]
For every $\ell\geq 1$, using the notation of Figure \ref{figure4}, we can see that $(-1)^\ell d^\ell c$ is the matrix (written blockwise)
\begin{equation}
\label{powers1}
\begin{blockarray}{cccc}
{} & {\quad I\quad} & {\quad L\quad} & {\quad L'\quad} \\
\begin{block}{c(c|c|c)}
J & (c_1\sigma)^\ell c_1 & 0 & (c_1\sigma)^\ell c_2 \\
\cline{2-4}
M & c_3\sigma(c_1\sigma)^{\ell-1}c_1 & 0 & c_3\sigma(c_1\sigma)^{\ell-1}c_2 \\
\cline{2-4}
M' & 0 & 0 & 0 \\
\end{block}
\end{blockarray}
\,,
\end{equation}
and this coincides with $(c\tau)^\ell c$.
Letting $\ell=2m$, this yields the formula claimed for $(x_\fs)^{2m+1}$ in the case of $m\geq 1$. The claimed formula is immediate if $m=0$.
\end{proof}

We consider the following extension-by-zero mappings:
\[
\xymatrix{
& \ar_{\eta_1:=\eta_{J\cup M,I\cup L'}^{\intq,\intp}\quad}[ld] \Mat_{J\cup M,I\cup L'}(\C) \ar^{\quad\eta_2:=\eta_{J\cup M,I\cup L'\\ }^{S,R}}[rd] & \\
\Mat_{\intq,\intp}(\C)=\Mat_{q,p}(\C) &  & \Mat_{S,R}(\C).}
\]
In addition, we consider the subspace $\Gamma:=\{\gamma\in \Mat_{J\cup M,I\cup L'}(\C):\varsigma \gamma\in\fn_{I\cup L'}^+,\ \gamma\varsigma\in\fn_{J\cup M}^+\}$ (where $\varsigma$ is as in (\ref{zeta-tau})).

\begin{lemma}
\label{L4.9}
The maps $\eta_1$ and $\eta_2$ restrict to linear isomorphisms
\[(\eta_1)|_{\Gamma}:\Gamma\stackrel{\sim}{\longrightarrow} \left\{ c\in \Mat_{q,p}(\C):\exists a,b,d\ \mbox{such that}\ \begin{pmatrix} a & b \\ c & d \end{pmatrix}\in\condir_\omega \right\}\]
and\[(\eta_2)|_\Gamma:\Gamma\stackrel{\sim}{\longrightarrow} \{z\in\Mat_{S,R}(\C):wz\in\fn_R^+,\ zw\in\fn_S^+\}=\{z\in\Mat_{S,R}(\C):wz\in\fn_R^+\cap({}^w\fn_S^+)\},\]
where $w=w_{\fs,+}$ is the bijection defined in (\ref{ws+}).
Moreover, for every $\gamma\in\Gamma$ such that $c=\eta_1(\gamma)$ and $z=\eta_2(\gamma)$, we have
\[\eta_1((\gamma\varsigma)^{\ell}\gamma)=(c\tau)^\ell c\quad\mbox{and}\quad \eta_2((\gamma\varsigma)^\ell\gamma)=(zw)^\ell z=w^{-1}(wz)^{\ell+1}\quad\mbox{for all $\ell\geq 0$},\]
where $\varsigma$ and $\tau$ are given by (\ref{zeta-tau}).
\end{lemma}

\begin{proof}
The space $\Gamma$ consists of the matrices
\begin{equation}
\label{gamma}
\gamma=\begin{blockarray}{ccc}
{} & {\quad I\quad} & {\quad L'\quad} \\
\begin{block}{c(c|c)}
J & c_1 & c_2 \\
\cline{2-3}
M & c_3 & c_4 \\ \end{block}
\end{blockarray}
\,\in \Mat_{J\cup M,I\cup L'}(\C)
\end{equation}
(written blockwise) such that
\begin{equation}
\label{zeta-gamma}
\varsigma \gamma=\begin{blockarray}{ccc}
{} & {\quad I\quad} & {\quad L'\quad} \\
\begin{block}{c(c|c)}
I & \sigma c_1 & \sigma c_2 \\
\cline{2-3}
L' & 0 & 0 \\ \end{block}
\end{blockarray}
\qquad\mbox{and}\qquad \gamma\varsigma=\begin{blockarray}{ccc}
{} & {\quad J\quad} & {\quad M\quad} \\
\begin{block}{c(c|c)}
J & c_1\sigma & 0 \\
\cline{2-3}
M & c_3\sigma & 0 \\ \end{block}
\end{blockarray}\end{equation}
are strictly upper triangular (before rearranging the rows and the columns).
Moreover, for every $\ell\geq 1$, we have
\begin{equation}
\label{powers2}
(\gamma\varsigma)^\ell\gamma=\begin{blockarray}{ccc}
{} & {\quad I\quad} & {\quad L'\quad} \\
\begin{block}{c(c|c)}
J & (c_1\sigma)^\ell c_1 & (c_1\sigma)^\ell c_2 \\
\cline{2-3}
M & c_3\sigma(c_1\sigma)^{\ell-1}c_1 & c_3\sigma(c_1\sigma)^{\ell-1}c_2 \\ \end{block}
\end{blockarray}\,.
\end{equation}

One can see from Figure \ref{figure4} that the space
\[
\mathcal{C}:=\left\{ c\in\Mat_{q,p}(\C):\exists a,b,d\ \mbox{such that}\ \begin{pmatrix} a & b \\ c & d \end{pmatrix}\in\condir_\omega\right\}
\]
consists of the matrices of the form
\[
c=\begin{blockarray}{cccc}
{} & {\quad I\quad} & {\quad L\quad} & {\quad L'\quad} \\
\begin{block}{c(c|c|c)}
J & c_1 & 0 & c_2 \\
\cline{2-4}
M & c_3 & 0 & c_4 \\
\cline{2-4}
M' & 0 & 0 & 0 \\
\end{block}
\end{blockarray}
\]
such that the matrices
\[
\begin{blockarray}{cccc}
{} & {\quad I\quad} & {\quad L\quad} & {\quad L'\quad} \\
\begin{block}{c(c|c|c)}
I & \sigma c_1 & 0 & \sigma c_2 \\
\cline{2-4}
L & 0 & 0 & 0 \\
\cline{2-4}
L' & 0 & 0 & 0 \\
\end{block}
\end{blockarray}\in\Mat_p(\C)
\quad\mbox{and}\quad
\begin{blockarray}{cccc}
{} & {\quad J\quad} & {\quad M\quad} & {\quad M'\quad} \\
\begin{block}{c(c|c|c)}
J & -c_1\sigma & 0 & 0 \\
\cline{2-4}
M & -c_3\sigma & 0 & 0 \\
\cline{2-4}
M' & 0 & 0 & 0 \\
\end{block}
\end{blockarray}\in\Mat_q(\C)
\]
are strictly upper triangular (before rearranging the rows and the columns). From the above description of $\Gamma$, these conditions are equivalent to having that $c=\eta_1(\gamma)$ for some $\gamma\in\Gamma$. Hence $\mathcal{C}=\eta_1(\Gamma)$, and this implies that $\eta_1$ restricts to a linear isomorphism $(\eta_1)|_\Gamma:\Gamma\to \mathcal{C}$.
Moreover, by comparing the expressions of $(c\tau)^\ell c$ and $(\gamma\varsigma)^\ell\gamma$ given in (\ref{powers1}) and (\ref{powers2}), respectively, one can see that the equality $\eta_1((\gamma\varsigma)^\ell\gamma)=(c\tau)^\ell c$ holds for all $\ell\geq 0$ whenever $\eta_1(\gamma)=c$.

We have $R=R'\cup I\cup L'$ and $S=J\cup M\cup S'$ where $R':=\{-i\}_{i=1}^t$
and $S':=\{q+j\}_{j=1}^{s'}$.
Let $z$ be an element in $\Mat_{S,R}(\C)$, and let us write it (blockwise) as
\begin{equation}
\label{z}
z=\begin{blockarray}{cccc}
{} & {\quad R'\quad} & {\quad I\quad} & {\quad L'\quad} \\
\begin{block}{c(c|c|c)}
J & z_1 & z_2 & z_3 \\
\cline{2-4}
M & z_4 & z_5 & z_6 \\
\cline{2-4}
S' & z_7 & z_8 & z_9 \\
\end{block}
\end{blockarray}\,.\end{equation}
The bijection $w=w_{\fs,+}:S\to R$ of (\ref{ws+}) can be written in the following matrix form
\[
w=\begin{blockarray}{cccc}
{} & {\quad J\quad} & {\quad M\quad} & {\quad S'\quad} \\
\begin{block}{c(c|c|c)}
R' & 0 & \sigma_0 & 0 \\
\cline{2-4}
I & \sigma & 0 & 0 \\
\cline{2-4}
L' & 0 & 0 & \sigma_0' \\
\end{block}
\end{blockarray}
\]
where the block $\sigma_0\in\Mat_{R',M}(\C)$ is yielded by the unique decreasing bijection $M\to R'$; this corresponds to a block with $1$'s on the antidiagonal and $0$'s elsewhere. The block $\sigma_0'\in\Mat_{L',S'}(\C)$ is defined in the same way.
We get
\[
wz=\begin{blockarray}{cccc}
{} & {\quad R'\quad} & {\quad I\quad} & {\quad L'\quad} \\
\begin{block}{c(c|c|c)}
R' & \sigma_0 z_4 & \sigma_0 z_5 & \sigma_0 z_6 \\
\cline{2-4}
I & \sigma z_1 & \sigma z_2 & \sigma z_3 \\
\cline{2-4}
L' & \sigma_0' z_7 & \sigma_0' z_8 & \sigma_0' z_9 \\
\end{block}
\end{blockarray}
\quad\mbox{and}\quad zw=\begin{blockarray}{cccc}
{} & {\quad J\quad} & {\quad M\quad} & {\quad S'\quad} \\
\begin{block}{c(c|c|c)}
J & z_2\sigma & z_1\sigma_0 & z_3\sigma_0' \\
\cline{2-4}
M & z_5\sigma & z_4\sigma_0 & z_6\sigma_0' \\
\cline{2-4}
S' & z_8\sigma & z_7\sigma_0 & z_9\sigma_0' \\
\end{block}
\end{blockarray}\,.
\]
Note that we have $i<j$ whenever $(i,j)\in R'\times(I\cup L')$. We have also $i<j$ whenever $(i,j)\in(J\cup M)\times S'$. We obtain that $(wz,zw)\in\fn_R^+\times \fn_S^+$ if and only if the following conditions are satisfied:
\[
\left\{\begin{array}{ll}
\mbox{$\bullet$ the blocks $z_1,z_4,z_7,z_8,z_9$ are zero}, \\
\parbox{14cm}{$\bullet$ the submatrices $\begin{blockarray}{ccc}
{} & {\quad I\quad} & {\quad L'\quad} \\
\begin{block}{c(c|c)}
I & \sigma z_2 & \sigma z_3 \\
\cline{2-3}
L' & 0 & 0 \\
\end{block}
\end{blockarray}$
\ and $\begin{blockarray}{ccc}
{} & {\quad J\quad} & {\quad M\quad} \\
\begin{block}{c(c|c)}
J & z_2\sigma & 0 \\
\cline{2-3}
M & z_5\sigma & 0 \\
\end{block}
\end{blockarray}$
\ are strictly upper triangular.}
\end{array}\right.
\]
Combining these observations with the description of the space $\Gamma$ given above, we conclude that the elements in the space $\mathcal{Z}:=\{z\in\Mat_{S,R}(\C):wz\in\fn_R^+,\ zw\in\fn_S^+\}$ are exactly of the form $z=\eta_2(\gamma)$ for $\gamma\in\Gamma$ (see (\ref{gamma})--(\ref{zeta-gamma})) and, therefore, $\eta_2$ restricts to an isomorphism $(\eta_2)|_\Gamma:\Gamma\to\mathcal{Z}$ as asserted.

Finally, let $z=\eta_2(\gamma)$ for $\gamma\in\Gamma$ written as in (\ref{gamma}).
In the notation of (\ref{z}), this means that $z_2=c_1$, $z_3=c_2$, $z_5=c_3$, $z_6=c_4$, and the other blocks of $z$ are zero. In view of the expression for $zw$ and $(\gamma\varsigma)^\ell\gamma$ given in (\ref{powers2}), this yields
\[(zw)^\ell z=\begin{blockarray}{cccc}
{} & {\quad R'\quad} & {\quad I\quad} & {\quad L'\quad} \\
\begin{block}{c(c|c|c)}
J & 0 & (c_1\sigma)^\ell c_1 & (c_1\sigma)^{\ell}c_2 \\
\cline{2-4}
M & 0 & c_3\sigma(c_1\sigma)^{\ell-1}c_1 & c_3\sigma(c_1\sigma)^{\ell-1}c_2 \\
\cline{2-4}
S' & 0 & 0 & 0 \\
\end{block}
\end{blockarray}\,=\eta_2((\gamma\varsigma)^\ell\gamma)\]
for all $\ell\geq 1$. The lemma is proved.
\end{proof}

\begin{proof}[Proof of Theorem \ref{T2}\,(2)]
Let $\Phi_\fk(\Xorbit_\omega)=\Nkorbit_{\lambda,\mu}$ and $\Phi_\fs(\Xorbit_\omega)=\Nsorbit_\Lambda$.
Let
\[x=\begin{pmatrix}a & b \\ c & d\end{pmatrix}\]
be a general element of $\mathcal{D}_\omega$, so that $x_\fk\in\Nkorbit_{\lambda,\mu}$ and $x_\fs\in\Nsorbit_\Lambda$
(see Lemma \ref{L4.1}).
It follows from Lemma \ref{L4.8} that, for every $\ell\geq 1$,
the number $\nrplus{\Lambda}{\ell}$ of $+$'s in the first $\ell$ columns of the signed Young diagram $\Lambda$ is equal to
\[\dim\ker a^{2m}\mbox{ if $\ell=2m$ is even},\quad\mbox{respectively}\quad \dim \ker(c\tau)^{2m} c\mbox{ if $\ell=2m+1$ is odd}\]
(with $\tau$ from (\ref{zeta-tau})).
If $\ell=2m$ is even, Theorem \ref{T2}\,(1) shows that this number is equal to the number $\nrboxes{\lambda}{2m}$ of boxes in the first $2m$ columns of $\lambda$.
This confirms the first assertion made in Theorem \ref{T2}\,(2)\,(a).

Now assume that $\ell=2m+1$ is odd. By Lemma \ref{L4.9}, we have $c=\eta_1(\gamma)$ with $\gamma\in\Gamma$. Moreover, we can assume that $z:=\eta_2(\gamma)$ is general in $\{z\in\Mat_{S,R}(\C):wz\in\fn_R^+\cap({}^w\fn_S^+)\}$,
with $w=w_{\fs,+}$ as in (\ref{ws+}). Hence, by Proposition \ref{P4.6}, this element $z$ is such that
\[
\dim \ker (wz)^{2m+1}=\nrboxes{\lambda'}{2m+1}
\]
with $\lambda'=\shape(\RS_1(w))$.
Note also that $\dim\ker(wz)^{2m+1}=\dim\ker(zw)^{2m}z$. In view of the previous observations, and by Lemma \ref{L4.9}, we get
\begin{eqnarray*}
\nrplus{\Lambda}{2m+1} & = & \dim \ker(c\tau)^{2m} c \\ & = & \card{L}+\dim \ker (\gamma\varsigma)^{2m}\gamma \\
 & = & \card{L}+\dim \ker(zw)^{2m}z-\card{R'} \\
 & = & s-t+\nrboxes{\lambda'}{2m+1}.
\end{eqnarray*}
This establishes the claim in Theorem \ref{T2}\,(2)\,(b) regarding the number of $+$'s. The formulas regarding the number of $-$'s stated in Theorem \ref{T2}\,(2)\,(a)--(b) can now be deduced, by invoking Lemma \ref{L4.3} and taking (\ref{4.7}) into account. The proof of Theorem \ref{T2}\,(2) is complete.
\end{proof}

\end{document}